\documentclass[12pt]{article}

\usepackage[latin1]{inputenc}
\usepackage{hyperref}
\usepackage{amsmath}
\usepackage{amsfonts}
\usepackage{amssymb}
\usepackage{mathrsfs}       
\usepackage{graphicx}
\usepackage[mathscr]{eucal}
\usepackage{subfigure}
\usepackage{enumerate}
\usepackage{dsfont}
\usepackage{amsthm} 
\usepackage{xcolor}
\usepackage{float}
\theoremstyle{plain}

\newtheorem{theorem}{Theorem}
\newtheorem{proposition}[theorem]{Proposition}
\newtheorem{definition}[theorem]{Definition}
\newtheorem{lemma}[theorem]{Lemma}

\def\mos{m}
\def\hum{h}
\def\MOS{M}
\def\HUM{H}

\def\controlm{u_m}

\def\campo{(g_{\mos},g_{\hum})}
\def\xx{\mathbf{x}}
\def\yy{\mathbf{y}}

\def\ff{\mathbf{f}}

\def\hamiltonian{\mathscr{L}}
\def\solucion{\mathscr{\HUM}}

\makeindex

\def\nucleo{\mathbb{V}(\overline{H})}
\def\constraintset{\mathbb{V}^{0}(\overline{H})}

\def\VV{\mathbb{V}}
\def\RR{\mathbb{R}}
\def\lb#1{\underline{#1}}  
\def\ub#1{\overline{#1}}  
\def\dub#1{\overline{\overline{#1}}}

\def\rectangle{[0, 1]\times [0, \overline{H}]}

\def\eqsepv{\; , \enspace}       
\def\eqfinv{\; ,}                
\def\eqfinp{\; .}

\newcommand{\bp}[1]{\big(#1\big)}                           

\title{Viable Control of an Epidemiological Model}

\def\email#1{#1}
\def\keywords#1{\textbf{Keywords:} #1}

\author{ Michel De Lara\footnote{%
Universit\'e Paris-Est, CERMICS (ENPC), 6-8 Avenue Blaise Pascal,
Cit\'e Descartes, 77455 Marne-la-Vall\'ee, France,
\email{delara@cermics.enpc.fr}}           
 \and
Lilian Sofia Sepulveda Salcedo \footnote{%
Universidad Aut\'onoma de Occidente, 
Km. 3 v\'ia Cali-Jamund\'i, Cali, Colombia,
\email{lssepulveda@uao.edu.co}}
}
\date{\today}

\begin{document}
\maketitle

\def\controlm{u}

\begin{abstract}
In mathematical epidemiology, epidemic control often aims at driving
the number of infected individuals to zero, asymptotically.
However, during the transitory phase, the 
number of infected can peak at high values. 
In this paper, we consider mosquito vector control in 
the Ross-Macdonald epidemiological model,
with the goal of capping the proportion of infected by dengue at the peak. 
We formulate this problem as one of control of a dynamical system
under state constraint. 
We allow for time-dependent fumigation rates to reduce the population
of mosquito vector, in order to maintain the proportion of infected 
individuals by dengue below a threshold for all times.
The so-called viability kernel is the set of initial states 
(mosquitoes and infected individuals) for which  
such a fumigation control trajectory exists.
Depending on whether the cap on the 
proportion of infected is low, high or medium,
we provide different expressions of the viability kernel.
We also characterize so-called viable policies that produce, 
at each time, a fumigation rate as a function of 
current proportions of infected humans and mosquitoes,
such that the proportion of infected humans remains 
below a threshold for all times. 
We provide a numerical application in the case of control
of a dengue outbreak in 2013 in Cali, Colombia.  

\end{abstract}


\tableofcontents
\pagebreak 

 \section{Introduction}

We consider vector control of the spread of an epidemic
transmitted by vector in a Ross-Macdonald model.
Our approach aims at controlling the number infected at the peak:
we look for a trajectory of time-dependent vector mortality rates 
that is able to maintain the number of infected 
individuals below a threshold for all times.
This approach differs from the widespread stationary control
strategies, based upon having control reproductive number stricly less
than one  to ensure convergence, 
and also from cost minimization optimal control ones.

Indeed, many studies on mathematical modeling of infectious diseases 
consist of analyzing the stability of the equilibria of a  
differential system (behavioral models such as SIR, SIS, SEIR
\cite{BrauerCastillo2006}).  
Those studies focus on asymptotic behavior and stability, 
generally leaving aside the transient behavior of the system, 
where the infection can reach high levels. 
In  many epidemiological models, a significant quantity is 
the ``basic reproductive number'' ${\cal R}_0$ 
which depends on parameters such as the transmission rate,
the death and birth rate, etc. 
Numerous works (see references in 
\cite{Hethcote:2000,Diekmann-Heesterbeek:2000})
exhibit conditions on ${\cal R}_0$ such that
the number of infected individuals tends towards zero. 
With this tool, different (time-stationary) 
management strategies of the propagation 
of the infection -- quarantine, vaccination, etc. -- are compared 
with respect to how they modify ${\cal R}_0$.
Thus, strategies are compared as to their capacity to drive 
the number of infected towards zero, focusing on asymptotics.
However, during the transitory phase, the 
number of infected can peak at high values. 

Other works deal with the whole trajectory, as in 
dynamic optimization where strategies 
are compared with respect to intertemporal costs and benefits
 \cite{Hethcote:1973}, \cite{Longini-Ackerman-Elveback:1978},
\cite{Hethcote:1988}, 
\cite{Kirschner-Lenhart-Serbin:1997},
\cite{Culshaw-Ruan-Spiteri:2004}, etc.
More recently, \cite{Hansen-Day:2010} studies controls that minimize 
the outbreak size (or infectious burden) under the assumption 
that there are limited control resources. 

Our approach focuses both on transitories and asymptotics,
in a robust way. Instead of aiming at an equilibrium or 
optimizing, we look for policies
able to maintain the infected individuals 
below a threshold for all times. 
To our knowledge, this approach is new in mathematical epidemiology.
We have only found it mentioned in passing in \cite{Hethcote:1973} as a
constraint -- bounding above the maximum number infected at the peak --
in a dynamic optimization problem, solved numerically.

In this paper, we formulate the problem as one of \emph{viability}.
In a nutshell, viability theory examines when and how can the state of a 
control system can be maintained within a given region for all times
\cite{Aubin1991}. 
Such problems of dynamic control under constraints also refer 
to invariance issues \cite{Clarkeetal:1995}.
In the control theory literature, problems of constrained control 
lead to the study of positively invariant sets, 
particularly ellipsoidal and polyhedral ones for linear systems 
(see\cite{Bitsoris:1988}, \cite{Gilbert-Tan:1991}, \cite{Gutman-Cwikel:1986} 
and the survey paper \cite{Blanchini:1999});
reachability of target sets or tubes for nonlinear discrete time
dynamics is examined in~\cite{bertsekas71minimax}.
In continuous time, 
such a viability approach has been applied to models related to 
the sustainable management of fisheries \cite{Bene-Doyen-Gabay:2001}, 
to viable strategies to ensure survival of some species \cite{Bonneuil-Saint-Pierre:2005}, to secure the prey predator system \cite{Bonneuil-Mullers:1997}, 
etc.
In discrete-time, different examples can be found in \cite{DeLara-Doyen:2008} 
for sustainable management applications of viability. 
\medskip

The paper is organized as follows. 
In Section~\ref{sec:planteamiento-problema}, 
we state the viability problem 
for the Ross-Macdonald model with control on the mosquito population
by requiring to keep the proportion of infected humans below 
a given threshold. The \emph{viability kernel} is the set of
initial states such that there exists at least one 
fumigation control trajectory such that the resulting 
proportion of infected humans remains below 
a given threshold for all times.
In Section~\ref{sec:caracterizacion-nucleo}, we provide
a characterization of the viability kernel depending on 
whether the cap for the proportion of infected humans 
is low, high or medium. 
We also discuss \emph{viable controls}.
We apply our theoretical results to the case of the dengue outbreak in 2013 in
Cali, Colombia. Thanks to numerical data provided by the
Municipal Secretariat of Public Health of Cali, we provide figures 
of viability kernels and of viable trajectories.

 \section{The viability problem}
 \label{sec:planteamiento-problema}

First, we present the Ross-MacDonald model. Second, we 
formulate the viability problem, which consists in 
capping the proportion of infected humans using the dynamics 
described by the Ross-MacDonald model. 
Then, we introduce the viability kernel and viability domains.
Finally, we describe viable equilibria, that are part of the
viability kernel. 

\subsection{The Ross-MacDonald model}

Different types of Ross-MacDonald models have been published 
\cite{Smith_et_al2007}. We choose the one in \cite{AndersonMay1992},
where both total populations (humans, mosquitoes) are normalized to~1
and divided between susceptibles and infected. 
The basic assumptions of the model are the following.
\begin{itemize}
  \item[i)] The human population ($N_{\hum}$) and the mosquito population
($N_{\mos}$) are closed and remain stationary.
  \item[ii)] Humans and mosquitoes are homogeneous in terms of 
susceptibility, exposure and contact.
  \item[iii)] The incubation period is ignored, in humans as in mosquitoes.
  \item[iv)] Death induced by the disease is ignored, 
in humans as in mosquitoes.
  \item[v)] Once infected, mosquitoes never recover.
  \item[vi)] Only susceptibles get infected, in humans as in mosquitoes.
  \item[vii)] Gradual immunity in humans is ignored.
\end{itemize}
Let $\mos(t)$ denote the \emph{proportion of infected mosquitoes} at time~$t$,
and $\hum(t)$ the \emph{proportion of infected humans} at time~$t$.
Therefore, $1-\mos(t)$ and $1-\hum(t)$ are the respective proportions 
of susceptibles. 
The Ross-MacDonald model is the following differential system
\begin{subequations}
\label{eq:model-Ross-Macdonald}
  \begin{eqnarray}
  \displaystyle\frac{d\mos}{dt}&=& 
\alpha\, p_{\mos}\,\hum(1-\mos)-\delta\, \mos\eqsepv
\label{eq:model-Ross-Macdonald_a} \\[5mm]
  \displaystyle\frac{d\hum}{dt}&=& 
\alpha\, p_{\hum}\, \frac{N_{\mos}}{N_{\hum}}\, \mos (1-\hum)-\gamma\, \hum\eqsepv
  \end{eqnarray}
\end{subequations}
where the parameters $\alpha$, $p_{\mos},\, p_{\hum}$, 
$\xi=\frac{N_{\mos}}{N_{\hum}}$, 
$\delta$ and $\gamma$ are given in Table~\ref{tabla:parametros-Ross}.

\begin{table}[h]
	\centering
	\begin{tabular}{|p{3cm}|p{9.5cm}|}\hline
		\bf{Parameter} & \bf{Description} \\ \hline \hline
$ \alpha \geq 0 $ & biting rate per time unit \\ \hline
$  \xi=N_{\mos}/N_{\hum} \geq 0 $  & number of female mosquitoes per human 
\\ \hline
$ 1 \geq p_{\hum} \geq 0 $ & probability of infection of a susceptible human \\
& by infected mosquito biting \\ \hline
$ 1 \geq p_{\mos} \geq 0 $ & 
probability of infection of a susceptible mosquito \\
& when biting an infected human \\ \hline
$\gamma \geq 0 $  & recovery rate for humans \\ \hline
$\delta \geq 0 $  & (natural) death rate for mosquitoes \\ \hline
\end{tabular}
\caption{Parameters of the Ross-MacDonald model~\eqref{eq:model-Ross-Macdonald}}
\label{tabla:parametros-Ross}
\end{table}

 \subsection{Capping the proportion of infected humans}
 \label{sec:sin-costos}

We turn the dynamical system~\eqref{eq:model-Ross-Macdonald} into a 
control system by replacing the natural death rate~$\delta$ for mosquitoes
in~\eqref{eq:model-Ross-Macdonald_a} 
by a piecewise continuous function~$\controlm(\cdot):t\rightarrow \controlm(t)\in [\lb{\controlm}, \ub{\controlm}]$, with $\lb{\controlm}=\delta$.
The function~$\controlm(\cdot)$ is the control on mosquito population 
that affects the mortality rate by fumigation of insecticides.
The upper bound~$\ub{\controlm}$ is the maximal fumigation mortality rate. 
For notational simplicity, we put
\begin{equation}
\label{eq:transmision-ross}
 A_{\mos}=\alpha p_{\mos} \eqsepv A_{\hum}=\alpha p_{\hum} \frac{M}{H}
\eqfinp
\end{equation}

Thinking about public health policies set by governmental entities, 
we impose the following constraint:
the proportion~$\hum(t)$ of infected humans 
must always remain below a threshold~$\ub{\HUM}$, where
\begin{equation}
  0< \ub{\HUM} < 1 \eqsepv
\end{equation}
represents the maximum tolerated proportion of infected humans. 

Therefore, the viability problem for the Ross-Macdonald model 
with control on the mosquito population is as follows. 
Given the control system 
\begin{subequations}
  \begin{align}
  \displaystyle\frac{d\mos}{dt}=& 
A_{\mos}\hum(t) (1-\mos(t))-\controlm(t)  \mos(t)  \eqsepv \\[5mm]
  \displaystyle\frac{d\hum}{dt}=& 
A_{\hum} \mos(t) (1-\hum(t))-\gamma  \hum(t) \eqsepv    
  \end{align}
\label{eq:model-control-vector}
\end{subequations}
determine if there exists a piecewise continuous 
control trajectory~$\controlm(\cdot)$ such that 
\begin{equation}
\label{eq:restriccion-control}
\controlm(\cdot):t \mapsto \controlm(t)\eqsepv
\lb{\controlm}\leq\controlm(t)\leq\ub{\controlm} \eqsepv 
\forall t \geq 0 \eqsepv 
\end{equation}
and such that the state trajectory given by~\eqref{eq:model-control-vector}
satisfies the viability constraint
\begin{equation}
\label{eq:restriccion-estado}
\hum(t) \leq \ub{\HUM} \eqsepv \forall t \geq 0 \eqfinp
\end{equation}
\label{eq:problema-viabilidad-recursos-ilimitados}

\subsection{Viability kernel and viability domains}

The solution to problem~\eqref{eq:problema-viabilidad-recursos-ilimitados} 
relies mainly on identifying the initial conditions, 
$(\mos(0), \hum(0))$, for the mosquitoes and infected humans, 
for which there exists a mortality rate due to fumigation, 
$\controlm(\cdot)$ like~\eqref{eq:restriccion-control}, 
such that the solution~\eqref{eq:model-control-vector}
starting from~$(\mos(0), \hum(0))$ 
satisfies~\eqref{eq:restriccion-estado}. 
Such set of initial conditions is called the \emph{viability kernel}.

\begin{definition}
\label{def:nucleo-viabilidad}
The set of initial conditions $(\mos_0, \hum_0)$ for which there 
exists at least one fumigation policy~\eqref{eq:restriccion-control} 
such that the solution to the system~\eqref{eq:model-control-vector},
with initial state $(\mos(0), \hum(0))=(\mos_0, \hum_0)$,
satisfies the constraint~\eqref{eq:restriccion-estado}
is called the \emph{viability kernel}.
We denote the {viability kernel} by $\nucleo$, that is, 
\begin{equation}
\nucleo=\left \{ (\mos_0, \hum_0) \left|
\begin{array}{c}
 \textit{there is  } \controlm(\cdot) 
\textit{as in }~\eqref{eq:restriccion-control} \\
  \textit{such that the solution to }~\eqref{eq:model-control-vector}\\
  \textit{which starts from }(\mos_0, \hum_0)\\
  \textit{satisfies the constraint }~\eqref{eq:restriccion-estado}       
\end{array} \right.
\right\} \eqfinp
\label{eq:nucleo-viabilidad}
\end{equation}
\end{definition}

The {viability kernel} is a subset of the \emph{constraint set}
\begin{equation}
\constraintset =\{(\mos, \hum)|0\leq\mos\leq 1, 0\leq\hum\leq \ub{\HUM} \}
=\rectangle \eqfinv
\label{eq:conjunto_de_restricciones}
\end{equation}
that is,
\begin{equation}
  \nucleo \subset \constraintset \eqfinp
\end{equation}

We define and present a geometric characterization of 
so-called viability domains of system~\eqref{eq:model-control-vector}, 
as they will be an important step to characterize 
the viability kernel.

\begin{definition}
\label{def:dominio-viable}
A subset $\VV$ of the set of states~$[0,1]\times [0,1]$
is said to be a \emph{viability domain} for 
the system~\eqref{eq:model-control-vector}--\eqref{eq:restriccion-control}
if there exists a control trajectory~$\controlm(\cdot)$ as 
in~\eqref{eq:restriccion-control} such that the solution 
to~\eqref{eq:model-control-vector}, which starts from
$(\mos(0), \hum(0))\in \VV$, 
remains within~$\VV$ for every $t\geq 0$.
\end{definition}

Viability domains are related to the viability kernel as follows.
\begin{theorem}[\cite{Aubin1991}]
The viability kernel is the largest viability domain within the constraint 
set.
\label{th:Aubin}
\end{theorem}

With system~\eqref{eq:model-control-vector}, 
we associate the vector field~$\campo$ given by the two components
\begin{subequations}
\begin{align}
g_{\mos}(\mos, \hum, \controlm) &= \label{eq:campo-control-vector_mos}
A_{\mos} \hum(1-\mos)-\controlm  \mos  \eqfinv \\
g_{\hum}(\mos, \hum) &= A_{\hum}  \mos (1-\hum)-\gamma  \hum \eqfinp
\label{eq:campo-control-vector_hum}
\end{align}
\label{eq:campo-control-vector}
\end{subequations}
The system~\eqref{eq:model-control-vector} is equivalent to
\begin{subequations}
\begin{align}
  \displaystyle\frac{d\mos}{dt}&= 
g_{\mos}\big(\mos(t), \hum(t), \controlm(t)\big) \eqsepv \\[5mm]
  \displaystyle\frac{d\hum}{dt}&=
g_{\hum}\big(\mos(t), \hum(t)\big) \eqfinp
\end{align}
\label{eq:campo-model-control-vector}
\end{subequations}

We now provide a geometric characterization of the viability domains of the 
system~\eqref{eq:campo-model-control-vector}
with control constraints~\eqref{eq:restriccion-control}
using the vector field~$\campo$. 
For this purpose, we first note that 
the system~\eqref{eq:campo-model-control-vector} 
is \emph{Marchaud} \cite{Aubin1991} because: 
\begin{itemize}
\item the constraint~\eqref{eq:restriccion-control} on the controls is 
written as $\controlm \in [\lb{\controlm},\ub{\controlm}]$, 
where $[\lb{\controlm},\ub{\controlm}]$ is closed;
\item the components of the vector field~$\campo$ 
in~\eqref{eq:campo-control-vector} are continuous; 
\item the vector field~$\campo$ and the set $[\lb{\controlm},\ub{\controlm}]$
have linear growth (because the partial derivatives of~$\campo$ are smooth 
and defined over the compact $[0,1]\times[0,1]$);
\item the set $\{ \campo(\mos,\hum,\controlm) \mid \controlm 
\in[\lb{\controlm},\ub{\controlm}] \}$ is convex, for all $(\mos,\hum)$, 
because $g_{\mos}(\mos,\hum,\controlm)$ is linearly dependent 
on the control~$\controlm$.
\end{itemize}

Second, we will use the following result to be found in~\cite{Aubin1991}
(Theorem~6.1.4, p.~203 and the remark p.~200).

\begin{proposition} 
For a Marchaud controlled system, a closed subset $\VV$ is viable if
the tangent cone at any point in $\VV$ contains at least one of the vectors
in the family generated by the vector field at this point 
when the control varies.
\label{pr:Marchaud}
\end{proposition}

In our case, we obtain the following geometric characterization 
of viability domains.

\begin{proposition}
\label{prop:geometria-dominio-viables}
Consider a closed subset~$\VV$ of $[0,1]\times [0,1]$. The set~$\VV$ is 
a viability domain for the system~\eqref{eq:model-control-vector}, if, 
whenever $(\mos, \hum)$ varies along the frontier~$\partial\VV$ of 
the set $\VV$,
there is a control $\controlm \in [\lb{\controlm},\ub{\controlm}]$ 
such that $(g_{\mos}(\mos,\hum,\controlm),g_{\hum}(\mos,\hum))$ 
is an inward-pointing vector, with respect to the set~$\VV$.
\end{proposition}

If the closed subset~$\VV$ has a piecewise smooth frontier~$\partial\VV$, 
it suffices --- for the set~$\VV$ to be a viability domain for the 
system~\eqref{eq:model-control-vector} ---
that the scalar product between the vector~$\campo$ and a normal 
(non zero) outward-pointing vector (with respect to the set~$\VV$) 
be lower than or equal to zero.

\subsection{Viable equilibria}
\label{Puntos_de_equilibrio}

Control systems display a family of equilibria, indexed by stationary 
decisions. Within them, the equilibria which satisfy the constraints 
are part of the viability kernel: they are said to be \emph{viable equilibria}.

\subsubsection*{Stationary control: mortality rate due to constant fumigation}
\label{fumigacion-constante}

Consider stationary mosquito control, that is, constant mortality rate due to 
fumigation: 
\begin{equation}
\controlm(t)=u_m \eqsepv \forall t \geq 0 \eqsepv \text{ with }
\lb{\controlm}\leq u_m\leq \ub{\controlm} \eqfinp
\end{equation}
The system~\eqref{eq:model-control-vector} has the disease free
equilibrium~$(0, 0)$ and, possibly, the endemic equilibrium point: 
\begin{equation}
\label{eq:endemico-control}
E^*_{u_m} = (\mos^*,\hum^*) =  \displaystyle \left(
\frac{A_{\mos} - \gamma u_m/A_{\hum} }{ A_{\mos} + u_m }, 
\frac{ A_{\hum} - \gamma u_m/A_{\mos} }{ A_{\hum} + \gamma }
\right )  \eqfinp
\end{equation}
Such point $E^*_{u_m}$ exists in $[0,1]^2$
and has global asymptotic stability when
\begin{equation}
\label{eq:existe-endemico-control}
0 < {A_{\mos} A_{\hum} - \gamma u_m} \eqfinp
\end{equation}
The proof relies on the Poincar\'e-Bendixson theorem,
excluding periodic orbits thanks to 
the Bendixson-Dulac criterion \cite{Wolfgang1998}.

The viable equilibrium points~$(\mos^*, \hum^*)$ are those for
which $\hum^*\leq\ub{\HUM}$. With~\eqref{eq:existe-endemico-control}, 
we deduce that the viable equilibrium points are the 
points~\eqref{eq:endemico-control} for which: 
\begin{equation}
\label{eq:endemico-viable}
0 < \frac{ A_{\hum} - \gamma u_m/A_{\mos} }{ A_{\hum} + \gamma }
\leq\ub{\HUM} \eqfinp
\end{equation}

\subsubsection*{Monotonicity properties}

The system \eqref{eq:model-control-vector} has monotonicity properties 
which will be practical for characterizing of the viability kernel. 

\begin{proposition}
\label{prop-cuasimonotono}
Let $\bp{\ub{\mos}(t), \ub{\hum}(t)}$ be the solution 
to~\eqref{eq:model-control-vector} when $\controlm(t)=\ub{\controlm}$. If
\begin{equation}
\ub{\mos}(0)\leq \mos(0) 
\eqsepv
\ub{\hum}(0)\leq \hum(0) 
\eqfinv  
\end{equation}
we have that 
\begin{equation}
\ub{\mos}(t)\leq\mos(t) 
\eqsepv 
\ub{\hum}(t)\leq\hum(t) 
\eqsepv \forall t > 0 \eqfinp
\end{equation}
\end{proposition}

\begin{proof}
Note that the components of vector field~$\campo$ 
in~\eqref{eq:campo-control-vector} are smooth, and that,
when $(\mos,\hum) \in [0,1]^2$, 
\begin{equation}
\frac{\partial g_{\mos}}{\partial \hum}=A_{\mos}  (1-\mos)\geq 0 \eqsepv 
\frac{\partial g_{\hum}}{\partial \mos}=A_{\hum}  (1-\hum)\geq 0 \eqfinp
\end{equation}
By Definition \ref{def:funcion-cuasimonotona} of 
Appendix~\ref{sec:resultados-teoricos}, $\campo$ is quasi monotonous 
in $(\mos, \hum)$ for any control $t \mapsto \controlm(t)$.

Denote by $(\ub{g_{\mos}},\ub{g_{\hum}})$ 
the vector fields when $\controlm=\ub{\controlm}$ and $\lb{\controlm}$, 
respectively, in~\eqref{eq:campo-control-vector}. Since
\( \ub{g_{\mos}}\leq g_{\mos}\) and 
\( \ub{g_{\hum}}\leq g_{\hum} \),
through the comparison Theorem \ref{teo:comparacion-edo-extension}, 
we obtain the following the result: if $\ub{\mos}(0)\leq \mos(0)$ and 
$\ub{\hum}(0)\leq \hum(0)$, then
\( \ub{\mos}(t)\leq\mos(t) \) and 
\( \ub{\hum}(t)\leq\hum(t) \), for every~$t\geq 0$. 
\end{proof}

\section{Characterization of the viability kernel}
\label{sec:caracterizacion-nucleo}

We will show that the characterization of the viability 
kernel~\eqref{eq:nucleo-viabilidad} depends on whether the upper 
limit~$\ub{\HUM}$ for the proportion of infected humans 
in~\eqref{eq:restriccion-estado} is low, high or medium. 

\begin{description}
\item[L)] 
When $\ub{\HUM}$ is \emph{low} in~\eqref{eq:restriccion-estado}, 
a strong constraint is put on the proportion of infected humans, 
and we will prove in~\S\ref{restriccion-fuerte}
that the viability kernel~\eqref{eq:nucleo-viabilidad} 
reduces to the origin~$\{(0,0)\}$. 

\item[H)]
When $\ub{\HUM}$ is \emph{high} in~\eqref{eq:restriccion-estado}, 
hence allowing the proportion of infected humans to be large, the viability
kernel is the entire constraint set~$\constraintset$ 
in~\eqref{eq:conjunto_de_restricciones}, as we will show 
in~\S\ref{restriccion-debil}. 

\item[M)]
 Finally, when $\ub{\HUM}$ is \emph{medium} in~\eqref{eq:restriccion-estado}, 
which is a more interesting case, the viability 
kernel~\eqref{eq:nucleo-viabilidad}  is a strict subset of the 
constraint set~$\constraintset$ 
in~\eqref{eq:conjunto_de_restricciones}, whose upper right frontier
is a smooth curve that we characterize. 
The proof of this result will be given in~\S\ref{restriccion-media}.
\end{description}

\subsection{When the infected humans upper bound is low}
\label{restriccion-fuerte}

When the imposed constraint~\eqref{eq:restriccion-estado} is strong, 
meaning that the upper bound~$\ub{\HUM}$ for the proportion of infected
 humans is low, the viability kernel~\eqref{eq:nucleo-viabilidad} 
reduces to the origin~$\{(0,0)\}$ as follows. 

\begin{proposition}
\label{prop:restriccion-fuerte}
If
\begin{equation}
\label{eq:cota-baja}
\ub{\HUM} < \frac{A_{\hum} - \gamma \ub{\controlm}/A_{\mos} }{A_{\hum} + \gamma}
\eqsepv 
\end{equation}
the viability kernel~\eqref{eq:nucleo-viabilidad} 
only consists of the origin:
\begin{equation}
\nucleo=\{(0,0)\} \eqfinp
\end{equation}
\end{proposition}

\begin{proof}
First, note that the state~$(0,0)$ is a viable equilibrium, 
as seen in~\S\ref{fumigacion-constante}. 
Second, if the initial conditions $(\mos(0), \hum(0))$
are taken outside $\{(0,0)\}$, let
us show that, for any $\controlm(\cdot)$ as in~\eqref{eq:restriccion-control}, 
the solution to~\eqref{eq:model-control-vector} violates 
the constraint~\eqref{eq:restriccion-estado}. 

Indeed, let $\ub{\mos}(t)$ and $\ub{\hum}(t)$ be solutions 
to~\eqref{eq:model-control-vector} when $\controlm(t)=\ub{\controlm}$
and with initial state $(\mos(0), \hum(0))$. 
The assumption~\eqref{eq:cota-baja} implies 
the condition~\eqref{eq:existe-endemico-control} that makes 
the endemic equilibrium point
\begin{equation}
\label{eq:endemico-control-maximo}
E^*_{\ub{\controlm}} = (\mos^*_{\ub{\controlm}},\hum^*_{\ub{\controlm}})
= \displaystyle \left(
\frac{A_{\mos} - \gamma u_m/A_{\hum} }{ A_{\mos} + u_m }, 
\frac{ A_{\hum} - \gamma u_m/A_{\mos} }{ A_{\hum} + \gamma }
\right)
\end{equation}
exist and display global asymptotic stability, 
as seen in~\S\ref{fumigacion-constante}. Hence, $\ub{\hum}(t)$ 
approaches~$\hum^*_{\ub{\controlm}}$ with $\ub{\HUM}<\hum^*_{\ub{\controlm}}$. 
So, by~\eqref{eq:cota-baja} and~\eqref{eq:endemico-control-maximo}, 
we have $\ub{\hum}(t)>\ub{\HUM}$ for all $t$ large enough. 

From Proposition~\ref{prop-cuasimonotono}, we have that 
$\ub{\hum}(t)\leq\hum(t)$, for every $t \geq 0$, so, 
for a significantly large~$t$, we have that 
$\ub{\HUM} <\ub{\hum}(t)\leq \hum(t)$. 

Therefore, for any initial condition outside $\{(0,0)\}$, 
for any  $\controlm(\cdot)$ as in~\eqref{eq:restriccion-control}, 
the solution to~\eqref{eq:model-control-vector} violates the 
constraint~\eqref{eq:restriccion-estado} for times~$t$ large enough,
hence at least for one time.
\end{proof}

\subsection{When the infected humans upper bound is high}
\label{restriccion-debil}

When the imposed constraint~\eqref{eq:restriccion-estado} is weak, 
meaning that the upper bound~$\ub{\HUM}$ for the proportion of infected humans 
is high, the viability kernel is the entire constraint 
set~$\constraintset$ in~\eqref{eq:conjunto_de_restricciones} as follows. 

\begin{proposition}
\label{prop:restriccion-debil}
If
\begin{equation}
\label{eq:cota-alta}
\frac{A_{\hum}}{A_{\hum} + \gamma} \leq \ub{\HUM} \eqfinv
\end{equation}
the constraint set~$\constraintset$ in~\eqref{eq:conjunto_de_restricciones}
is strongly invariant and is, therefore, the viability kernel:
\begin{equation}
\nucleo=\constraintset 
=\rectangle \eqfinp
\end{equation}
\end{proposition}

\begin{proof}
Consider the vector field~$\campo$ described 
in~\eqref{eq:campo-control-vector}, which corresponds to the 
system~\eqref{eq:model-control-vector}. We will study the vector~$\campo$ 
along the four faces of the rectangle $\constraintset$ and 
we will prove that, for any control, 
$\campo$ is an inward-pointing vector, with respect to the set~$\constraintset$ 
(that is, the vector~$\campo$ belongs to the tangent cone, which is closed). 
Thanks to Proposition~\ref{pr:Marchaud} (in fact, a time-varying extension), 
this suffices to prove that $\constraintset$ is strongly invariant.
\begin{itemize}
\item 
For any state~$(\mos,\hum)$ on the vertical half-line~$\mos=0$
and $\hum \geq 0$, we have that: 
\begin{equation}
g_{\mos}(0,\hum,\controlm)=A_{\mos}\hum \geq 0 \eqsepv  
g_{\hum}(0,\hum)= -\gamma \hum \leq 0 \eqfinp
 \end{equation}
\begin{figure}
\centering
\includegraphics[width=10cm,height=6cm]{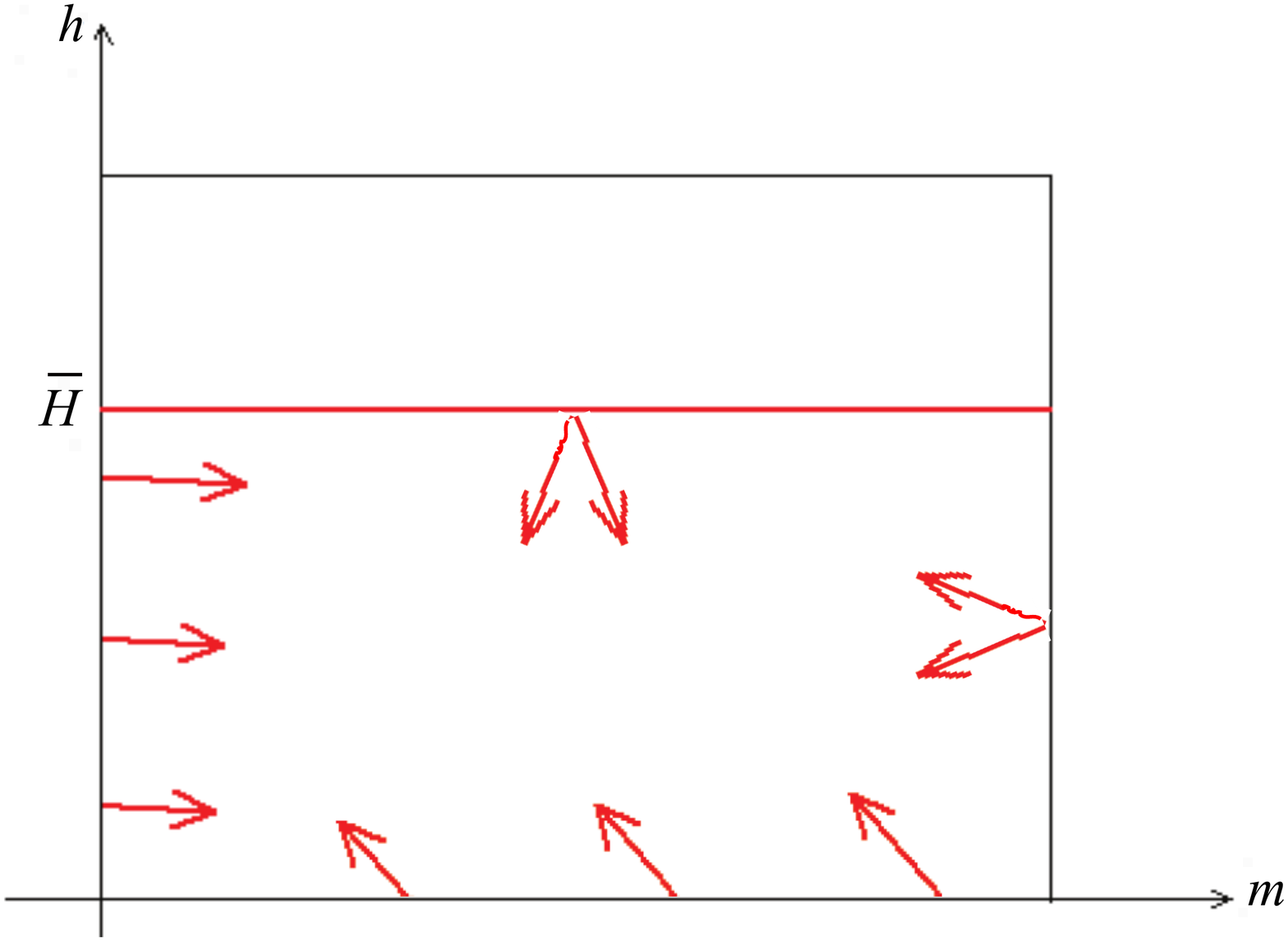}
\caption{Vector field~$\campo$ on the frontier of the constraint set~$\constraintset$}
\label{fig:frontera-conjunto-restriccion}
\end{figure}
Therefore, for any control $\controlm \in [\lb{\controlm},\ub{\controlm}]$, 
the vector~$\campo$ always points towards the inside of~$\constraintset$
(see the left part of Figure~\ref{fig:frontera-conjunto-restriccion}).

\item For any state~$(\mos,\hum)$ on the horizontal 
half-line~$\hum=0$ and $\mos \geq 0$, we have that:
\begin{equation}
g_{\mos}(0,\hum,\controlm)=-\controlm \mos \leq 0 \eqsepv  
g_{\hum}(0,\hum)= A_{\hum} \mos \geq 0 \eqfinp
\end{equation}
Therefore, for any control $\controlm \in [\lb{\controlm},\ub{\controlm}]$, 
the vector~$\campo$ always points towards the inside of~$\constraintset$
(see the bottom part of Figure~\ref{fig:frontera-conjunto-restriccion}).

\item For any state~$(\mos,\hum)$ on the vertical line~$\mos=1$, 
we have that
\begin{equation}
 g_{\mos}(1,\hum,\controlm)= -\controlm \leq -\lb{\controlm} \leq 0 \eqfinp
\end{equation}
Therefore, for any control $\controlm \in [\lb{\controlm},\ub{\controlm}]$, 
the vector~$\campo$ always points to the left of the line~$\mos=1$. Hence, 
the vector~$\campo$ always points towards the inside of~$\constraintset$
(see the right part of Figure~\ref{fig:frontera-conjunto-restriccion}).

\item Finally, for any state~$(\mos,\hum)$ on the horizontal
half-line~$\hum=\ub{\HUM}$ and $\mos \leq 1$,  we have that
\begin{equation}
g_{\hum}(\mos, \ub{\HUM})=A_{\hum} \mos (1-\ub{\HUM})-\gamma  \ub{\HUM} 
\leq A_{\hum} (1-\ub{\HUM})-\gamma  \ub{\HUM} \leq 0 \eqfinv
\end{equation}
by~\eqref{eq:cota-alta} 
(which is equivalent to $A_{\hum} (1-\ub{\HUM}) \leq \gamma \ub{\HUM} $). 
Therefore, for any control~$\controlm \in [\lb{\controlm},\ub{\controlm}]$, 
the vector~$\campo$ always points below the line~$\hum=\ub{\HUM}$. 
Hence, the vector~$\campo$ always points towards the inside of~$\constraintset$
(see the top part of Figure~\ref{fig:frontera-conjunto-restriccion}).

\end{itemize}
Therefore, the constraint set~$\constraintset$ is strongly invariant. 
So, the viability kernel is $\constraintset$ because every trajectory 
starting from~$\constraintset$ remains in~$\constraintset$, hence satisfies 
the constraint~\eqref{eq:restriccion-estado}.  
\end{proof}

\subsection{When the infected humans upper bound is medium}
\label{restriccion-media}

When the constraint is medium, that is, not too weak or too strong, 
then the upper bound for the proportion of infected humans is medium,
and the viability kernel is in-between the origin~$\{(0,0)\}$ and 
the entire constraint set~$\constraintset$ 
in~\eqref{eq:conjunto_de_restricciones}.
We introduce
\begin{equation}
\ub{\MOS}=\displaystyle\frac{\gamma \ub{\HUM}}{A_{\hum}(1-\ub{\HUM})}
\eqfinv 
\label{eq:hat_mos}
\end{equation}
the value $\mos$ on the line~$\hum=\ub{\HUM}$ where the 
component~$g_{\hum}(\mos, \hum)$ in~\eqref{eq:campo-control-vector_hum} is zero, 
that is,
\begin{equation}
 g_{\hum}(\ub{\MOS},\ub{\HUM}) = 
\ub{\MOS}(1-\ub{\HUM})-\gamma\ub{\HUM} = 0 \eqfinp
\label{eq:carac_hat_mos}
\end{equation}

\begin{proposition}
\label{prop:restriccion-media}
\begin{subequations}
If the upper bound for the proportion of infected humans is medium,
that is, if 
\begin{equation}
  \frac{A_{\hum} - \gamma \ub{\controlm}/A_{\mos} }{%
A_{\hum} + \gamma } 
< \ub{\HUM} < \frac{A_{\hum}}{A_{\hum} + \gamma} \eqfinv
\label{eq:cota-media_a}
\end{equation}
where
\begin{equation}
  0<\frac{A_{\hum} - \gamma \ub{\controlm}/A_{\mos} }{%
A_{\hum} + \gamma } \eqfinv
\label{eq:cota-media_b}   
\end{equation}
\label{eq:cota-media}    
\end{subequations}
we have that $\ub{\MOS}$, as defined in~\eqref{eq:hat_mos} is such that 
$\ub{\MOS}<1$, 
and the differential equation
\begin{subequations}
\begin{equation}
\label{eq:frontera-derecha}
-g_{\mos}(\mos, \solucion(\mos), \ub{\controlm}) \solucion'(\mos)+g_{\hum}(\mos, \solucion(\mos))=0 \eqfinv
\end{equation}
with initial condition
\begin{equation}
\label{eq:condicion-inicial-frontera}
\solucion(\ub{\MOS})=\ub{\HUM}  \eqfinv
\end{equation}
\end{subequations}
has a unique nonnegative solution, and this solution is of the form
\begin{equation}
\solucion : [\ub{\MOS}, \MOS_{\infty} ] \to [0,\ub{\HUM}] \text{ with }
\begin{cases}
\text{either \quad}  \ub{\MOS} < \MOS_{\infty} < 1 \text{ and } 
\solucion(\MOS_{\infty})=0 \eqfinv\\
\text{or \quad} \MOS_{\infty} =1 \eqfinp 
\end{cases}
\label{eq:solucion_Y}
\end{equation}
The viability kernel~\eqref{eq:nucleo-viabilidad} is $\nucleo=$
\begin{equation}
\left ([0, \ub{\MOS}]\times [0, \ub{\HUM}]\right )
\bigcup \left\{(\mos, \hum)\Big | \ub{\MOS}\leq \mos\leq \MOS_{\infty}
\eqsepv 0\leq\hum\leq \solucion(\mos)\right \}  \eqfinp
\end{equation}
\end{proposition}

On the Figure~\ref{fig:nucleo1-2013}, we display three viability kernels
when the constraint is medium (corresponding to three values for 
the maximum tolerated proportion~$\ub{\HUM}$ of infected humans). 
\begin{figure}
\centering
\includegraphics[width=13cm,height=7.7cm]{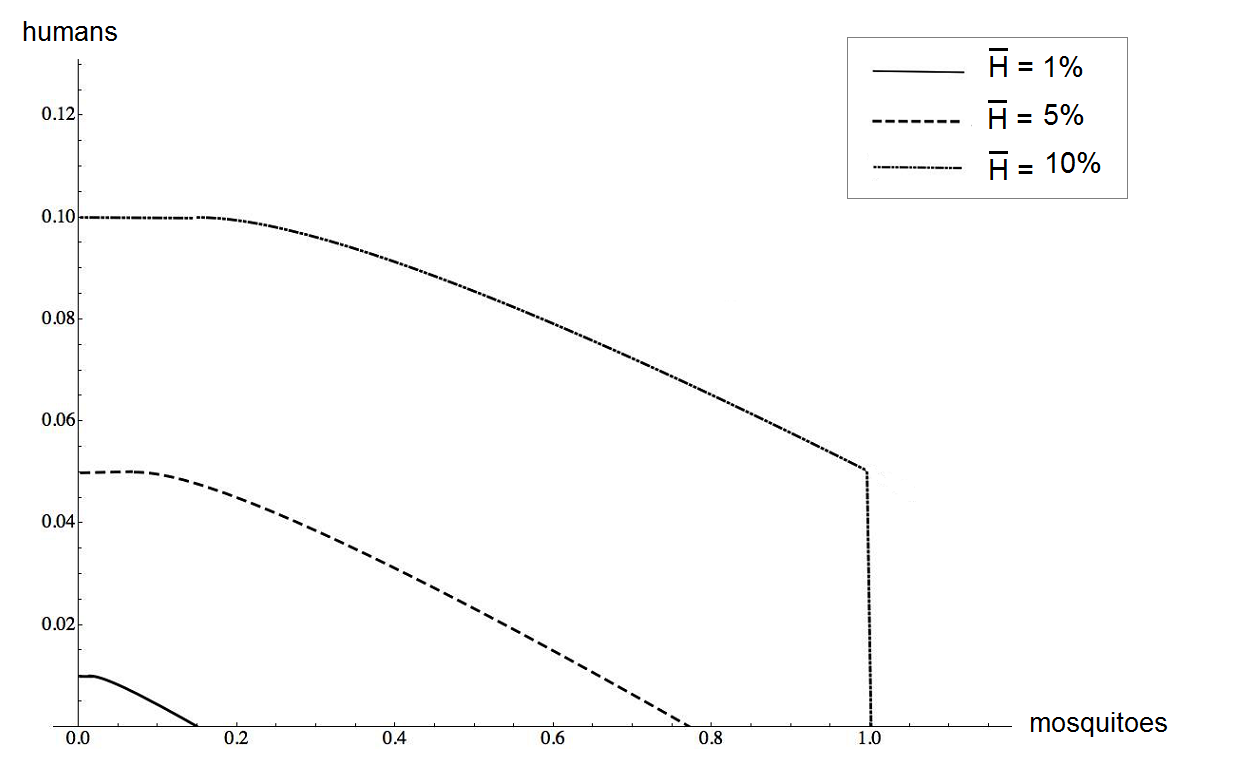}
\caption{Three viability kernels for the Ross-MacDonald 
model~\eqref{eq:model-Ross-Macdonald}, when the constraint is medium 
(corresponding to three values for 
the maximum tolerated proportion~$\ub{\HUM}$ of infected humans). 
Parameters are adjusted to the 2013 dengue outbreak in Cali, Colombia.
}
\label{fig:nucleo1-2013}
\end{figure}

The proof consists of four lemmas. 
In Lemma~\ref{lema:solucion-frontera-derecha}, we describe the solution 
to~\eqref{eq:frontera-derecha}--\eqref{eq:condicion-inicial-frontera}
(Lemma~\ref{lema:frontera-derecha} provides additional information, that 
will be useful).
Lemma~\ref{lema:dominio-viable} shows that the set
\begin{equation}
\label{eq:nucleo-medio}
\VV=\left ([0, \ub{\MOS}]\times [0, \ub{\HUM}]\right )
\bigcup \left\{(\mos, \hum)\Big | \ub{\MOS}\leq \mos\leq \MOS_{\infty}
\eqsepv 0\leq\hum\leq \solucion(\mos)\right \} 
\end{equation}
is a viability domain.  
Finally, in Lemma~\ref{lema:mayor-dominio-viable} we prove
that the set~$\VV$, defined in~\eqref{eq:nucleo-medio}, 
is the largest viability domain within the constraint 
set~\eqref{eq:conjunto_de_restricciones}, hence is the viability kernel
by Theorem~\ref{th:Aubin}. 

\begin{lemma}
\label{lema:solucion-frontera-derecha}
When the inequalities~\eqref{eq:cota-media} are fulfilled, 
there exists a function~$\solucion$, solution
to the differential equation~\eqref{eq:frontera-derecha} which satisfies 
the initial condition~\eqref{eq:condicion-inicial-frontera} and has the 
form~\eqref{eq:solucion_Y}. The function 
$\solucion: [\ub{\MOS},\MOS_{\infty}]  \to [\solucion(\MOS_{\infty}),\ub{\HUM}]$ 
is a strictly decreasing one-to-one mapping. 
\end{lemma}

\begin{proof}
We conduct the proof in five steps.
\begin{enumerate}
\item 
\label{it:coeficiente-mos<0}
First, we note that 
\begin{eqnarray*}
g_{\mos}(\ub{\MOS}, \ub{\HUM}, \ub{\controlm}) &=&
A_{\mos} \ub{\HUM} (1-\ub{\MOS}) -\ub{\controlm} \ub{\MOS}
\text{ by~\eqref{eq:campo-control-vector_mos} } \\
&=& A_{\mos} \ub{\HUM} - (A_{\mos} \ub{\HUM}+\ub{\controlm})\ub{\MOS} \\
&=& A_{\mos} \ub{\HUM} - (A_{\mos} \ub{\HUM}+\ub{\controlm})
\frac{\gamma \ub{\HUM}}{A_{\hum}(1-\ub{\HUM})} 
 \text{ by definition~\eqref{eq:hat_mos} of~ } \ub{\MOS} \\
&=& \frac{\ub{\HUM}}{A_{\hum}(1-\ub{\HUM})} [A_{\mos}A_{\hum}(1-\ub{\HUM})
- \gamma (A_{\mos} \ub{\HUM}+\ub{\controlm}) ]  \\
&=& \frac{\ub{\HUM}}{A_{\hum}(1-\ub{\HUM})} 
[A_{\mos}A_{\hum}- \gamma \ub{\controlm} - 
(A_{\mos}A_{\hum}+\gamma A_{\mos}) \ub{\HUM} ] \\
&<& 0 \text{ by~\eqref{eq:cota-media_b}.}
\end{eqnarray*}
We easily deduce the following property, that will be useful later:
\begin{equation}
\mos \geq \ub{\MOS} \eqsepv \hum \leq \ub{\HUM} \Rightarrow 
g_{\mos}({\MOS}, {\HUM}, \ub{\controlm}) \leq
g_{\mos}(\ub{\MOS}, \ub{\HUM}, \ub{\controlm}) < 0 \eqfinp 
\label{eq:isocline_mos}
\end{equation}

\item 
Second, we show that there is a local solution to the differential
equation~\eqref{eq:frontera-derecha}--\eqref{eq:condicion-inicial-frontera}. 
Indeed, in the neighborhood of $(\ub{\MOS},\solucion(\ub{\MOS}))=
(\ub{\MOS},\ub{\HUM})$, the coefficient 
$g_{\mos}(\mos, \hum, \ub{\controlm})$ of~$\solucion'(\mos)$ 
in~\eqref{eq:frontera-derecha} is negative, as we just saw it in
the previous item~\ref{it:coeficiente-mos<0}. 
Hence, in a neighborhood of $\ub{\MOS}$, 
we can write~\eqref{eq:frontera-derecha} as 
\begin{equation}
  \label{eq:frontera-derecha-equivalente}
\solucion'(\mos)=\frac{g_{\hum}(\mos, \solucion(\mos))}{%
g_{\mos}(\mos, \solucion(\mos), \ub{\controlm})}=
\frac{A_{\hum} \mos(1-\solucion(\mos))-\gamma \solucion(\mos)}{%
A_{\mos} \solucion(\mos) (1-\mos)-\ub{\controlm} \mos} \eqfinp
\end{equation}
Applying the Cauchy-Lipschitz theorem 
to~\eqref{eq:frontera-derecha-equivalente}--\eqref{eq:condicion-inicial-frontera}, we know that there exists a unique local $C^1$ solution~$\solucion$ 
defined on an interval~$I$ around~$\ub{\MOS}$. We put
\begin{equation}
I_+= I \cap [\ub{\MOS}, +\infty[ \eqfinp  
\end{equation}

\item 
\label{it:estrictamente_decreciente}
Third, we show that the local solution $\solucion: I_+ \to \RR$ is 
strictly decreasing in the neighborhood of~$\ub{\MOS}$. 
For this purpose, we will study the sign of  
$g_{\hum}(\mos, \solucion(\mos))$ for $m \approx \ub{\MOS}$. We have that
\begin{eqnarray*}
  g_{\hum}(\mos, \solucion(\mos)) &=& 
g_{\hum}(\mos, \solucion(\mos))-g_{\hum}(\ub{\MOS},\ub{\HUM}) 
\text{ by~\eqref{eq:carac_hat_mos}} \\
&=& A_{\hum} \mos(1-\solucion(\mos))-\gamma \solucion(\mos) -
A_{\hum} \ub{\MOS}(1-\ub{\HUM})+\gamma\ub{\HUM} \\ 
&=& A_{\hum} (\mos-\ub{\MOS})(1-\ub{\HUM})
+ (A_{\hum} \mos +\gamma) (\ub{\HUM}-\solucion(\mos)) \eqfinp
\end{eqnarray*}
We have that \( \solucion(\mos)-\ub{\HUM} =o(\mos-\ub{\MOS}) \)
when $\mos \to \ub{\MOS}$. 
Indeed, the function~$\solucion$ is $C^1$ 
and such that $\solucion'(\ub{\MOS}) = 0$,
by~\eqref{eq:frontera-derecha-equivalente} 
since $g_{\hum}(\ub{\MOS},\ub{\HUM}) = 0$ by~\eqref{eq:carac_hat_mos}. 
Therefore, we deduce that 
\begin{equation}
g_{\hum}(\mos, \solucion(\mos)) =   A_{\hum} (\mos-\ub{\MOS})(1-\ub{\HUM}) 
+  o(\mos-\ub{\MOS}) \eqfinp 
\end{equation}

Hence, when  $\mos > \ub{\MOS}$ and $\mos\approx\ub{\MOS}$, 
we have that $g_{\hum}(\mos, \solucion(\mos))>0$. 
Therefore, by~\eqref{eq:frontera-derecha-equivalente} and 
item~\ref{it:coeficiente-mos<0}, the function~$\solucion: I_+ \to \RR$ 
is strictly decreasing in a neighborhood of~$\ub{\MOS}$. 

\item
\label{it:estrictamente_decreciente_bis} 
Fourth, we show that the local solution $\solucion: I_+ \to \RR$ 
is strictly decreasing. Indeed, let us suppose the contrary: there exists
$\mos \in I_+ $, $\mos > \ub{\MOS}$, such that $\solucion'(\mos) = 0$. 
We denote by~$\tilde{\mos}$ the smallest of such~$\mos$.  
By item~\ref{it:estrictamente_decreciente}, we know that 
$\solucion'(\mos) <0$ in the neighborhood of~$\ub{\MOS}$ 
(except at $\ub{\MOS}$), so that $\tilde{\mos} > \ub{\MOS}$. 
By definition of~$\tilde{\mos}$, we have that $\solucion'(\mos) <0$ 
for $\mos \in ]\ub{\MOS},\tilde{\mos}[$. Hence, 
$ \solucion(\tilde{\mos}) < \solucion(\ub{\MOS})=\ub{\HUM} $, so that 
\begin{eqnarray*}
g_{\hum}(\tilde{\mos}, \solucion(\tilde{\mos})) 
&=& A_{\hum} \tilde{\mos}(1-\solucion(\tilde{\mos}))
-\gamma \solucion(\tilde{\mos}) 
\text{ by~\eqref{eq:campo-control-vector}} \\ 
& > & A_{\hum} \ub{\MOS}(1-\ub{\HUM}) -\gamma \ub{\HUM} \\ 
& =  & 
g_{\hum}(\ub{\MOS},\ub{\HUM}) = 0 \text{ by~\eqref{eq:carac_hat_mos}.} 
\end{eqnarray*}
As a consequence, by the differential
equation~\eqref{eq:frontera-derecha-equivalente} and 
item~\ref{it:coeficiente-mos<0}, we obtain that $\solucion'(\tilde{\mos}) > 0$.  We arrive at a contradiction, since $\solucion'(\tilde{\mos}) = 0$
by definition of~$\tilde{\mos}$. Therefore, $\tilde{\mos}$ does not exist and 
the local solution $\solucion: I_+ \to \RR$ is strictly decreasing. 

\item 
\label{item:forma-solucion-frontera-derecha}
Finally, we show that the local solution $\solucion: I_+ \to \RR$ 
is such that~\eqref{eq:solucion_Y} holds true. 
For this purpose, we consider two cases. 

\begin{itemize}
\item If, for every $\mos\in I_+$, $ 0 <  \solucion(\mos)$, 
we show that $I_+ = [\ub{\MOS}, +\infty[$. Indeed, as $\solucion$ decreases, 
we have $\mos\in I_+ \Rightarrow
0< \solucion(\mos)\leq \solucion(\ub{\MOS})=\ub{\HUM} $.
Hence, we deduce that the solution $\solucion(\mos)$ to the 
equation~\eqref{eq:frontera-derecha} exists over 
$m \in  [\ub{\MOS}, +\infty[$. We denote $\MOS_{\infty}=1$. 

\item If there exists an $\mos\in I_+$ such that $\solucion(\mos)=0 $, 
we denote by~$\tilde{\mos}$ the smallest of such~$\mos$.  
We have that $\solucion(\tilde{\mos})=0$ 
and, since $\solucion$ decreases, we have that 
$\solucion(\tilde{\mos})=0\leq \solucion(\mos)\leq 
\solucion(\ub{\MOS})=\ub{\HUM}$ for every 
$\mos\in[\ub{\MOS}, \tilde{\mos}]$.

We consider two subcases, as illustrated in 
Figures~\ref{fig:curva1-frontera} and \ref{fig:curva2-frontera}.
\begin{itemize}
\item 
If $\tilde{\mos} < 1 $,  we have that $\solucion(\tilde{\mos})=0$ 
and that $0\leq \solucion(\mos)\leq\ub{\HUM}$ for all $\mos\in[\ub{\MOS},
\tilde{\mos}]$. 
We denote $\MOS_{\infty}=\tilde{\mos} < 1$.
\begin{figure}
\centering
\includegraphics[width=10cm,height=6cm]{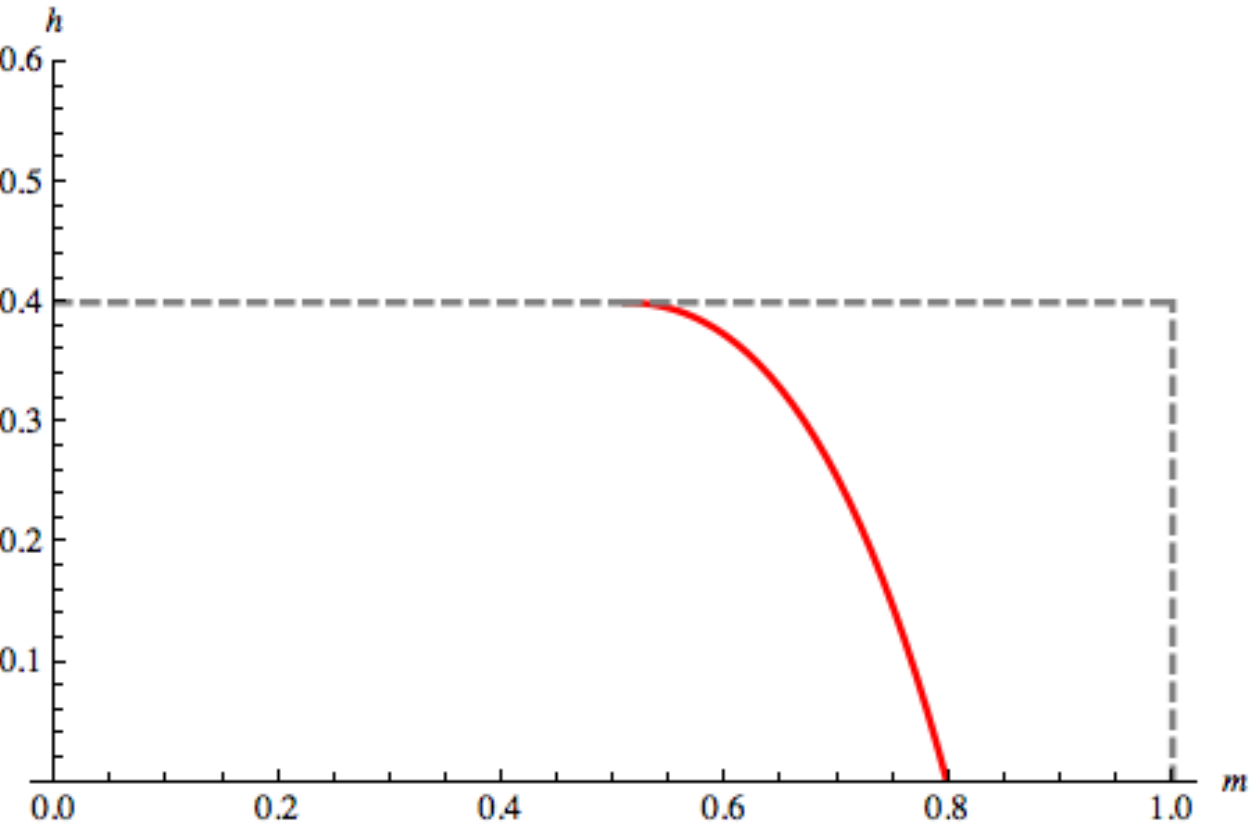}
\caption{Case $\MOS_{\infty}<1$. Solution graph for the differential equation 
\eqref{eq:frontera-derecha}--\eqref{eq:condicion-inicial-frontera},
with parameters $A_{\hum}=$0.31066, $A_{\mos}=$0.02906, $\gamma=0.1$, $\ub{\controlm}=$0.03733 and $\ub{\HUM}=$ 0.4}
\label{fig:curva1-frontera}
\end{figure}

\item 
if $\tilde{\mos} \geq 1 $,  we deduce that 
$0\leq \solucion(\mos)\leq\ub{\HUM}$ for every $\mos\in[\ub{\MOS}, 1]$. 
We denote $\MOS_{\infty}=1$.
\begin{figure}
\centering
\includegraphics[width=10cm,height=6cm]{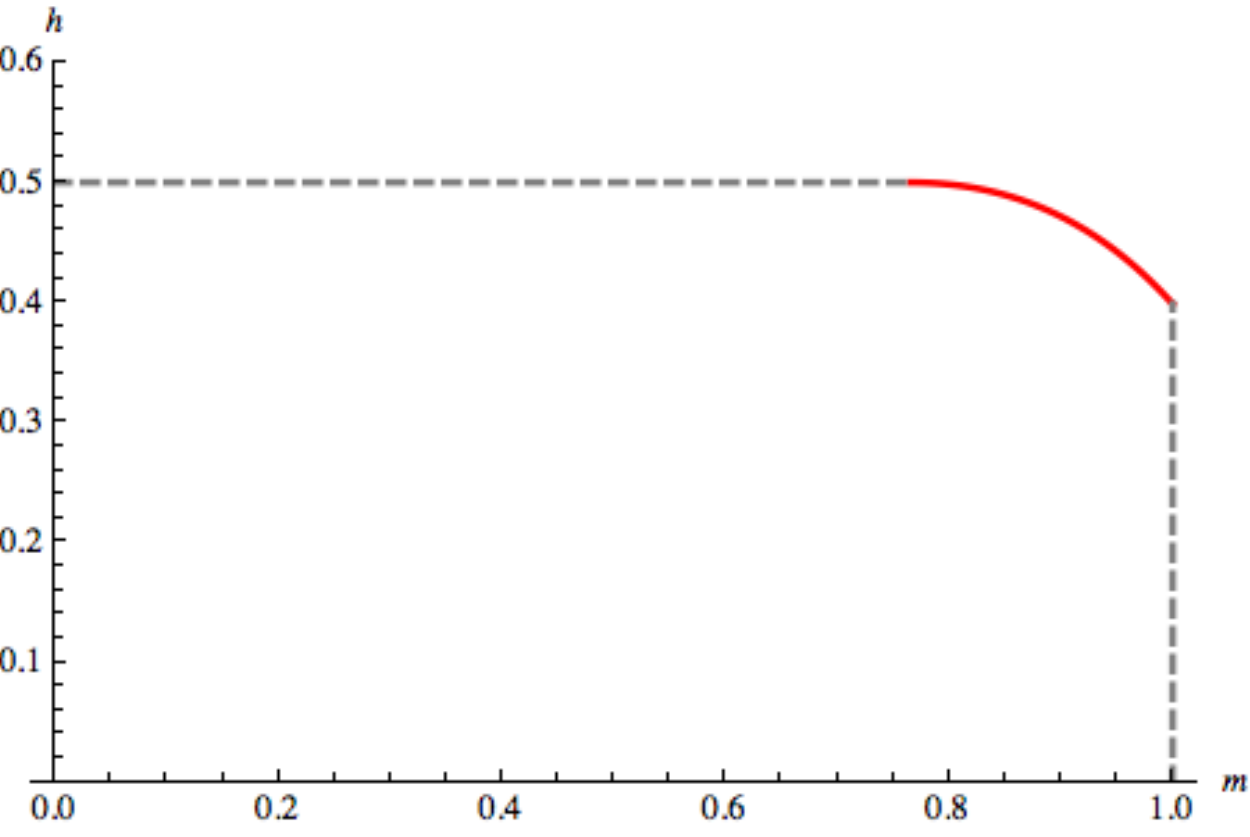}
\caption{Case $\MOS_{\infty}=1$. Solution graph for the differential equation 
\eqref{eq:frontera-derecha}--\eqref{eq:condicion-inicial-frontera},
with parameters $A_{\hum}=$0.31066, $A_{\mos}=$0.02906, $\gamma=0.1$, $\ub{\controlm}=$0.03733 and $\ub{\HUM}=$ 0.5}
\label{fig:curva2-frontera}
\end{figure}
\end{itemize}
\end{itemize}

\end{enumerate}

Hence, we conclude that the solution to the differential 
equation~\eqref{eq:frontera-derecha}, which satisfies the initial 
condition~\eqref{eq:condicion-inicial-frontera}, is unique, 
has the form~\eqref{eq:solucion_Y} and is strictly decreasing.
As $\solucion(\ub{\MOS})=\ub{\HUM}$, the function 
$\solucion: [\ub{\MOS},\MOS_{\infty}]  \to [\solucion(\MOS_{\infty}),\ub{\HUM}]$ 
is a strictly decreasing one-to-one mapping. 
\end{proof}

We know show that the curve generated by the solution~$\solucion$
is an orbit.
\begin{lemma}
\label{lema:frontera-derecha}
The curve 
\begin{equation}
\{ \bp{ \mos, \solucion(\mos) } \mid \mos \in [\ub{\MOS}, \MOS_{\infty} ] \}
\label{eq:curve_frontera-derecha}
\end{equation}
is an orbit of the vector field the vector~$\campo$,
for the control~$\ub{\controlm}$. Indeed, for all 
\begin{equation}
\dub{\mos}_{0} \in [\ub{\MOS}, \MOS_{\infty} ] \text{ and }
\dub{\hum}_{0}=\solucion(\dub{\mos}_{0}) \in 
[\solucion(\MOS_{\infty}),\ub{\HUM}] \eqfinv
\label{eq:on_the_curve_frontera-derecha}
\end{equation}
there exists~$T \geq 0$ such that the orbit of 
the trajectory $ t \in [0,T] \mapsto \bp{\dub{\mos}(t), \dub{\hum}(t)}$ --- 
the solution to~\eqref{eq:model-control-vector} when 
$\controlm(t)=\ub{\controlm}$ and when the starting point is 
$\bp{\dub{\mos}(0), \dub{\hum}(0)}=(\dub{\mos}_{0}, \dub{\hum}_{0})$ ---
is included in the curve~\eqref{eq:curve_frontera-derecha}. 
In addition, 
\begin{equation}
\dub{\mos}(T)=\ub{\MOS} \text{ and } \dub{\hum}(T)=\ub{\MOS} \eqfinp
\label{eq:Tsuite_egalite_and}
\end{equation}
\end{lemma}

\begin{proof}
When $\dub{\mos}_{0}=\ub{\MOS}$, then 
\( \solucion(\dub{\mos}_{0}) = \solucion(\ub{\MOS})=\ub{\HUM} \)
by~\eqref{eq:condicion-inicial-frontera}, and the point
\( \bp{\dub{\mos}_{0}, \dub{\hum}_{0}} = (\ub{\MOS},\ub{\HUM}) \)
indeed belongs to the curve~\eqref{eq:condicion-inicial-frontera}.
Therefore~$T=0$.
\medskip

Now, we suppose that $\dub{\mos}_{0}>\ub{\MOS}$, so that 
\( \solucion(\dub{\mos}_{0}) < \ub{\HUM} \) because~$ \solucion$
is strictly decreasing by Lemma~\ref{lema:solucion-frontera-derecha}.
We define
\begin{equation}
T = \inf \{t \geq 0 \mid \dub{\mos}(t) < \ub{\MOS} \text{ or }
\dub{\hum}(t) > \ub{\HUM} \} \eqfinv
\label{eq:T}
\end{equation}
which is such that $T>0$, since 
\( \dub{\mos}(0)=\dub{\mos}_{0}> \ub{\MOS} \) and
\( \dub{\hum}(0)=\solucion(\dub{\mos}_{0}) < \ub{\HUM} \).
\medskip

We prove that $T< +\infty$. 
Since the inequality~\eqref{eq:cota-media_b} is fulfilled by assumption, 
the condition~\eqref{eq:existe-endemico-control} is satisfied, so that
the endemic equilibrium point~$(\mos^*_{\ub{\controlm}},\hum^*_{\ub{\controlm}})$
in~\eqref{eq:endemico-control-maximo} exists.
By definition of an equilibrium point, we have that 
\( g_{\mos}(\mos^*_{\ub{\controlm}},\hum^*_{\ub{\controlm}},\ub{\controlm})
=0 \). Therefore, as the left hand side of inequality~\eqref{eq:cota-media_a} 
can be restated as $\hum^*_{\ub{\controlm}} < \ub{\HUM} $,
we deduce that $ \mos^*_{\ub{\controlm}} < \ub{\MOS} $, 
by~\eqref{eq:isocline_mos}.
As seen in~\S\ref{fumigacion-constante}, the endemic equilibrium 
point~\eqref{eq:endemico-control-maximo} displays global asymptotic 
stability. 
Therefore, $\dub{\mos}(t)\to \mos^*_{\ub{\controlm}}$ when $t \to +\infty$. 
As $\mos^*_{\ub{\controlm}} < \ub{\MOS}$,
we deduce that there is a time~$t$ such that $\dub{\mos}(t) <
\ub{\MOS}$. As a consequence, $T$ as defined in~\eqref{eq:T} is finite:
$T< +\infty$. 
\medskip

We now study the trajectory 
$ t \in [0,T] \mapsto \bp{\dub{\mos}(t), \dub{\hum}(t)}$. 
By definition~\eqref{eq:T} of~$T$, we have that 
\begin{subequations}
  \begin{equation}
\dub{\mos}(T)=\ub{\MOS} \text{ or } \dub{\hum}(T)=\ub{\MOS}
\label{eq:Tsuite_egalite}
  \end{equation}
and that 
\begin{equation}
  \dub{\mos}(t) \geq \ub{\MOS} \text{ and } \dub{\hum}(t) \leq \ub{\HUM}
\eqsepv \forall t \in [0,T] \eqfinp
\label{eq:Tsuite_inegalite}
\end{equation}
\end{subequations}
By~\eqref{eq:isocline_mos} and~\eqref{eq:Tsuite_inegalite}, we deduce that 
\[
\frac{d \dub{\mos}(t) }{dt} = 
g_{\mos}(\dub{\mos}(t),\dub{\hum}(t),\ub{\controlm}) < 0 
\eqsepv \forall t \in [0,T] \eqfinp 
\]
As $ \dub{\mos}_{0} \leq \MOS_{\infty}$ 
by~\eqref{eq:on_the_curve_frontera-derecha}, 
we deduce that $ \dub{\mos}(t) \leq \dub{\mos}_{0} \leq \MOS_{\infty}$, 
for all $t \in [0,T]$.
Together with $\dub{\mos}(t) \geq \ub{\MOS}$, we obtain that 
$ \dub{\mos}(t) \in [\ub{\MOS},\MOS_{\infty}]$, for all $t \in [0,T]$.
Therefore, $\solucion(\dub{\mos}(t))$ is well defined since 
$\solucion: [\ub{\MOS},\MOS_{\infty}]  \to [\solucion(\MOS_{\infty}),\ub{\HUM}]$ 
by~\eqref{eq:solucion_Y}. We have, for all $t \in [0,T]$, 
\begin{align*}
\frac{d}{dt}[\dub{\hum}(t)-\solucion(\dub{\mos}(t))]= &
  g_{\mos}(\dub{\mos}(t),\dub{\hum}(t),\ub{\controlm}) -
\solucion'(\dub{\mos}(t))
g_{\hum}(\dub{\mos}(t),\dub{\hum}(t)) \\
& \text{by~\eqref{eq:model-control-vector} and 
\eqref{eq:campo-control-vector} } \\
=& 0 \text{ by~\eqref{eq:frontera-derecha}.}
\end{align*}
Therefore, for all $t \in [0,T]$, 
\begin{equation}
  \dub{\hum}(t)-\solucion(\dub{\mos}(t))=
\dub{\hum}(0)-\solucion(\dub{\mos}(0))=0
 \text{ by~\eqref{eq:on_the_curve_frontera-derecha},}
\end{equation}
so that the trajectory 
$ t \in [0,T] \mapsto \bp{\dub{\mos}(t), \dub{\hum}(t)}$ 
is included in the curve~\eqref{eq:curve_frontera-derecha}.
\medskip

As a particular case, we have that 
\( \dub{\hum}(T)-\solucion(\dub{\mos}(T)) = 0 \).
As $\dub{\mos}(T)=\ub{\MOS}$ or $\dub{\hum}(T)=\ub{\MOS}$
by~\eqref{eq:Tsuite_egalite}, we conclude that
\begin{equation*}
\dub{\mos}(T)=\ub{\MOS} \text{ and } \dub{\hum}(T)=\ub{\MOS} \eqfinv
\end{equation*}
since the function 
$\solucion: [\ub{\MOS},\MOS_{\infty}]  \to [\solucion(\MOS_{\infty}),\ub{\HUM}]$ 
is a strictly decreasing one-to-one mapping.
We have proven~\eqref{eq:Tsuite_egalite_and}. 
\end{proof}

\begin{lemma}
\label{lema:dominio-viable}
When the inequalities~\eqref{eq:cota-media} are fulfilled, 
the set~$\VV$, defined in~\eqref{eq:nucleo-medio}, is a viability domain 
for the system~\eqref{eq:model-control-vector}--\eqref{eq:restriccion-control}. 
\end{lemma}

\begin{proof}
Let the \emph{Hamiltonian}~$\hamiltonian$ be defined, for every vector 
$(n_{\mos}, n_{\hum})$, every state $(\mos,\hum)$ and control~$\controlm$, by
\begin{equation}
\label{eq:hamiltoniano}
\hamiltonian(\mos,\hum,n_{\mos},n_{\hum},\controlm)
=g_{\mos}(\mos,\hum,\controlm) n_{\mos}+g_{\hum}(\mos,\hum) n_{\hum}
\eqfinp
\end{equation}
The Hamiltonian is the scalar product between the vectors~$\campo$ 
and~$(n_{\mos}, n_{\hum})$.

We will check that, for any given point along the 
piecewise-smooth frontier of the set~$\VV$, defined in~\eqref{eq:nucleo-medio}, 
there is at least one  control~$\controlm \in
[\lb{\controlm},\ub{\controlm}]$ such that 
the value~\eqref{eq:hamiltoniano} of the Hamiltonian 
is lower than or equal to zero when 
$(n_{\mos}, n_{\hum})$ is a normal outward-pointing vector 
(with respect to the set~$\VV$). 
For kink points between two smooth parts, we will do the same
but with the cone generated by two normal outward-pointing vectors, 
corresponding to each of the smooth parts. 

We will divide the frontier of 
the set~$\VV$, defined in~\eqref{eq:nucleo-medio}, 
in five or seven parts 
(see Figures~\ref{fig:curva1-frontera} and \ref{fig:curva2-frontera})
as follows. 


\begin{enumerate}[(a)]
\item  On the horizontal segment~$\{(\mos, 0)| 0<\mos < \MOS_{\infty}\}$, 
on the vertical segment~$\{(0, \hum)| 0<\hum < \ub{\HUM}\}$ 
 and at the points~$(0,0)$ and $(0, \ub{\HUM})$,
all vectors~$\campo$ point towards the inside of
the set~$\VV$, defined in~\eqref{eq:nucleo-medio}.
Indeed, it suffices to copy the proof of 
Proposition~\ref{prop:restriccion-debil}. 

\item Along the segment $\{(\mos, \ub{\HUM})| 0<\mos<\ub{\MOS}\}$, 
a normal outward-pointing vector is
\begin{equation*}
\left (
\begin{array}{c}
n_{\mos}\\ 
n_{\hum}
\end{array}
\right )=\left (
\begin{array}{c}
0\\
1
\end{array}
\right ) \eqfinp
\end{equation*}
Therefore, for any control $\controlm \in
[\lb{\controlm},\ub{\controlm}]$,
the value~\eqref{eq:hamiltoniano} of the Hamiltonian is
\begin{eqnarray*}
	\hamiltonian(\mos,\ub{\HUM},n_{\mos},n_{\hum},\controlm)
= g_{\hum} (\mos,\ub{\HUM}) \times 1 <  0 \eqfinv
\end{eqnarray*}
by~\eqref{eq:carac_hat_mos} and because $\mos<\ub{\MOS}$.

\item Along the curve 
$\{(\mos,\solucion(\mos))|\ub{\MOS}<\mos<\MOS_{\infty}\}$,
a normal outward-pointing vector is  
\begin{equation*}
\left (
\begin{array}{c}
n_{\mos}\\
n_{\hum}
\end{array}
\right )= 
\left (
\begin{array}{c}
-\solucion'(\mos)\\
1
\end{array}
\right ) 
 \eqfinp
\end{equation*}
Therefore, for the control~$\ub{\controlm}$, 
the value~\eqref{eq:hamiltoniano} of the Hamiltonian is
\[
\hamiltonian(\mos,\ub{\HUM},n_{\mos},n_{\hum},\ub{\controlm})=
-g_{\mos}(\mos, \solucion(\mos), \ub{\controlm}) \solucion'(\mos)
+g_{\hum}(\mos, \solucion(\mos)) \times 1 = 0  \eqfinv
\]
because the function~$\solucion$ is the solution to the 
differential equation~\eqref{eq:frontera-derecha}. 
 
\item The point $(\ub{\MOS},\ub{\HUM})$ is a kink, to which is attached
a cone of normal vectors generated by the two following
normal outward-pointing vectors.
\begin{itemize}
\item A normal outward-pointing vector to the horizontal segment 
$\{(\mos, \ub{\HUM})| 0\leq\mos\leq\ub{\MOS}\}$ 
at the kink point $(\ub{\MOS},\ub{\HUM})$ is
\begin{equation*}
\left (
\begin{array}{c}
n_{\mos}\\
n_{\hum}
\end{array}
\right )=
 \left (
\begin{array}{c}
0\\
1 
\end{array}
\right ) \eqfinp
\end{equation*}
Therefore, for any control $\controlm \in [\lb{\controlm},\ub{\controlm}]$, 
the value~\eqref{eq:hamiltoniano} of the Hamiltonian is
\[
\hamiltonian(\ub{\MOS},\ub{\HUM},n_{\mos},n_{\hum},\controlm) = 
g_{\hum}(\ub{\MOS}, \ub{\HUM}) \times 1 =0 
\text{ by~\eqref{eq:carac_hat_mos}.}
\]

\item A normal outward-pointing vector to the curve 
$\{(\mos,\solucion(\mos))|\ub{\MOS}\leq\mos\leq\MOS_{\infty}\}$ 
at the kink point $(\ub{\MOS},\ub{\HUM})$ is
\begin{equation*}
\left (
\begin{array}{c}
n_{\mos}\\
n_{\hum}
\end{array}
\right )= 
\left (
\begin{array}{c}
-\solucion'({\MOS})\\
1
\end{array}
\right ) = \left (
\begin{array}{c}
0\\
1
\end{array}
\right )
\text{ by~\eqref{eq:frontera-derecha-equivalente} and~\eqref{eq:carac_hat_mos}.}
\end{equation*}
Therefore, for any control $\controlm \in
[\lb{\controlm},\ub{\controlm}]$, 
the value~\eqref{eq:hamiltoniano} of the Hamiltonian is
\[
\hamiltonian(\ub{\MOS},\ub{\HUM},n_{\mos},n_{\hum},\controlm) = 
g_{\hum}(\ub{\MOS}, \ub{\HUM}) \times 1 =0 
\text{ by~\eqref{eq:carac_hat_mos}.}
\]
\end{itemize}
We conclude that, for any control $\controlm \in
[\lb{\controlm},\ub{\controlm}]$, 
the value~\eqref{eq:hamiltoniano} of the Hamiltonian is
zero for any combination of the two normal outward-pointing vectors above. 
As a consequence, all vectors~$\campo$ point towards the inside of
the set~$\VV$, defined in~\eqref{eq:nucleo-medio},
at the kink point $(\ub{\MOS},\ub{\HUM})$.

\item When $\MOS_{\infty}<1$ in~\eqref{eq:solucion_Y} 
(see Figure~\ref{fig:curva1-frontera}), there only remains to consider 
the kink point $(\MOS_{\infty}, 0)$.
\begin{itemize}
\item A normal outward-pointing vector to the horizontal segment 
$\{(\mos, 0)| 0\leq\mos\leq\MOS_{\infty}\}$ at the 
kink point $(\MOS_{\infty},0)$ is
\begin{equation*}
\left (
\begin{array}{c}
n_{\mos}\\ 
n_{\hum}
\end{array}
\right )=\left (
\begin{array}{c}
1 \\
0
\end{array}
\right ) \eqfinp
\end{equation*}
Therefore, for any control $\controlm \in
[\lb{\controlm},\ub{\controlm}]$, 
the value~\eqref{eq:hamiltoniano} of the Hamiltonian is
\[
\hamiltonian(\MOS_{\infty},0,n_{\mos},n_{\hum},\controlm) = 
g_{\mos}(\MOS_{\infty}, 0,\controlm) \times 1 = -\controlm \MOS_{\infty} \leq 
- \lb{\controlm}\MOS_{\infty} \leq 0 \eqfinp 
\]

\item A normal outward-pointing vector to the curve 
$\{(\mos,\solucion(\mos))|\ub{\MOS}\leq\mos\leq\MOS_{\infty}\}$ 
at the kink point $(\MOS_{\infty},0)$ is
\begin{equation*}
\left (
\begin{array}{c}
n_{\mos}\\
n_{\hum}
\end{array}
\right )= 
\left (
\begin{array}{c}
-\solucion'(\MOS_{\infty})\\
1
\end{array}
\right )
 \eqfinp
\end{equation*}
Therefore, for the control $\ub{\controlm}$, 
the value~\eqref{eq:hamiltoniano} of the Hamiltonian is
\begin{eqnarray*}
\hamiltonian(\MOS_{\infty},0,n_{\mos},n_{\hum},\ub{\controlm}) & = &
-g_{\mos}(\MOS_{\infty}, 0,\ub{\controlm})\solucion'(\MOS_{\infty}) +
g_{\hum}(\MOS_{\infty}, 0)  \times 1 \\
&=& 0 \text{ by~\eqref{eq:frontera-derecha}.}
\end{eqnarray*}
\end{itemize}
We conclude that, for the control~$\ub{\controlm}$, 
the value~\eqref{eq:hamiltoniano} of the Hamiltonian is
zero for any combination of the two normal outward-pointing vectors above. 
As a consequence, the vector~$\campo$, for the control~$\ub{\controlm}$,
points towards the inside of
the set~$\VV$, defined in~\eqref{eq:nucleo-medio}, 
at the kink point $(\MOS_{\infty},0)$.

\item When $\MOS_{\infty}=1$ in~\eqref{eq:solucion_Y} 
(see Figure~\ref{fig:curva2-frontera}), we have to consider 
the vertical segment $\{(1,\hum)| 0\leq\hum < \solucion(1) \}$.
A normal outward-pointing vector to the vertical segment 
$\{(1,\hum)| 0\leq\hum < \solucion(1) \}$ has the form
\begin{equation*}
\left (
\begin{array}{c}
n_{\mos}\\ 
n_{\hum}
\end{array}
\right )=\left (
\begin{array}{c}
1\\
0
\end{array}
\right ) \eqfinp
\end{equation*}
Therefore, for any control $\controlm \in
[\lb{\controlm},\ub{\controlm}]$, 
the value~\eqref{eq:hamiltoniano} of the Hamiltonian is
\[
	\hamiltonian(1,\hum,n_{\mos},n_{\hum},\controlm) =
g_{\mos}(1, \hum, \controlm) \times 1 = -\controlm \leq -
\lb{\controlm} \leq 0 \eqfinp
\]
As a consequence, all vectors~$\campo$ point towards the inside of
the set~$\VV$, defined in~\eqref{eq:nucleo-medio}, along 
the vertical segment $\{(1,\hum)| 0\leq\hum < \solucion(1) \}$.

\item When $\MOS_{\infty}=1$ in~\eqref{eq:solucion_Y} 
(see Figure~\ref{fig:curva2-frontera}), there remains to consider 
the kink point $(1, \solucion(1))$.

\begin{itemize}
\item A normal outward-pointing vector to the vertical segment 
$\{(1,\hum)| 0\leq\hum\leq \solucion(1) \}$ 
at the kink point $(1, \solucion(1))$ is 
\begin{equation*}
\left (
\begin{array}{c}
n_{\mos}\\ 
n_{\hum}
\end{array}
\right )=\left (
\begin{array}{c}
1 \\
0
\end{array}
\right ) \eqfinp
\end{equation*}
Therefore, for any control $\controlm \in
[\lb{\controlm},\ub{\controlm}]$, 
the value~\eqref{eq:hamiltoniano} of the Hamiltonian is
\[
\hamiltonian(1,\solucion(1),n_{\mos},n_{\hum},\controlm) =
g_{\mos}(1, \solucion(1), \controlm) \times 1 = 
-\controlm \leq -\lb{\controlm} < 0 \eqfinp
\]

\item A normal outward-pointing vector to the curve 
$\{(\mos,\solucion(\mos))|\ub{\MOS}<\mos<\MOS_{\infty}\}$ 
at the kink point $(1, \solucion(1))$ is 
\begin{equation*}
\left (
\begin{array}{c}
n_{\mos}\\
n_{\hum}
\end{array}
\right )= 
\left (
\begin{array}{c}
-\solucion'(1)\\
1
\end{array}
\right )
 \eqfinp
\end{equation*}
Therefore, for the control $\ub{\controlm}$, 
the value~\eqref{eq:hamiltoniano} of the Hamiltonian is
\begin{eqnarray*}
\hamiltonian(1,\solucion(1),n_{\mos},n_{\hum},\ub{\controlm}) & = &
-g_{\mos}(1, \solucion(1),\ub{\controlm})\solucion'(1) + 
g_{\hum}(1, \solucion(1))  \times 1 \\
&=& 0 \text{ by~\eqref{eq:frontera-derecha}.}
\end{eqnarray*}
\end{itemize}
We conclude that, for the control~$\ub{\controlm}$, 
the value~\eqref{eq:hamiltoniano} of the Hamiltonian is
zero for any combination of the two normal outward-pointing vectors above. 
As a consequence, the vector~$\campo$, for the control~$\ub{\controlm}$,
points towards the inside of
the set~$\VV$, defined in~\eqref{eq:nucleo-medio}, 
at the kink point $(1, \solucion(1))$.
\end{enumerate}

From Proposition~\ref{prop:geometria-dominio-viables}, 
we can conclude that $\VV$, defined in~\eqref{eq:nucleo-medio}, 
is a viability domain. 
\end{proof}

\begin{lemma}
\label{lema:mayor-dominio-viable}
When the inequalities~\eqref{eq:cota-media} are fulfilled, 
the set~$\VV$, defined in~\eqref{eq:nucleo-medio}, 
is the largest viability domain for 
the system~\eqref{eq:model-control-vector}, 
within the constraint set~\eqref{eq:conjunto_de_restricciones}.
\end{lemma}

\begin{figure}
\centering
\includegraphics[width=14cm,height=8cm]{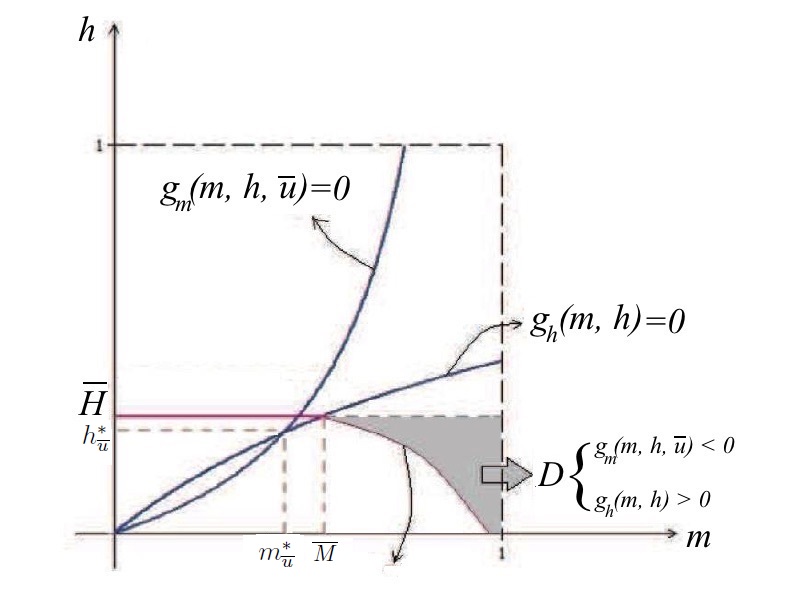}
\caption{The shaded portion of the graph corresponds to the complement of the viability kernel~$\nucleo$ with respect to the constraint set~$\constraintset$, when $\MOS_{\infty}<1$}
\label{fig:complemento-nucleo-medio}
\end{figure}

\begin{proof}
We prove that, for any control trajectory~$\controlm(\cdot)$ 
as in~\eqref{eq:restriccion-control}, 
the trajectory  $ t \mapsto \bp{{\mos}(t), {\hum}(t)}$ solution 
of~\eqref{eq:model-control-vector} and that starts from a point 
\begin{equation}
\bp{{\mos}(0), {\hum}(0)}= (\mos_{0}, \hum_{0}) 
\in \constraintset\backslash\VV
\label{eq:outsideVV}
\end{equation}
does not satisfy the constraint~\eqref{eq:restriccion-estado} for 
at least one $t>0$. 
In Figure~\ref{fig:complemento-nucleo-medio}, 
the set~$\constraintset\backslash\VV$ is represented by 
the shaded part (in the case when $\MOS_{\infty}<1$).
\medskip


First, we examine the point $(\mos_{0}, \hum_{0}) 
\in \constraintset\backslash\VV $.
On the one hand, as $(\mos_{0}, \hum_{0})\in \constraintset$,
we have that $ 0 \leq \hum_{0}\leq \ub{\HUM}$.
On the other hand, since $(\mos_{0}, \hum_{0})\notin \VV$, where
the set~$\VV$ is defined in~\eqref{eq:nucleo-medio}, 
we deduce from Lemma~\ref{lema:solucion-frontera-derecha},
and especially from~\eqref{eq:solucion_Y}, that 
\begin{align}
(\mos_{0}, \hum_{0}) \in \constraintset\backslash\VV & \iff  
\ub{\MOS} < \mos_{0} \leq 1 \eqsepv 0 \leq \hum_{0}\leq \ub{\HUM} 
\label{eq:constraintsetbackslashVV}  \\ 
& \text{ and } \nonumber 
\begin{cases}
\text{either \quad} & \mos_{0} > \MOS_{\infty} \eqfinv\\
\text{or \quad} & \mos_{0} \leq \MOS_{\infty} \text{ and } 
\solucion(\mos_{0})<\hum_{0} 
\eqfinp 
\end{cases}
\end{align}
Notice that, if $\hum_{0}= \ub{\HUM} $, then 
\begin{align*}
\frac{d {\hum}(t) }{dt}_{|t=0} =& g_{\hum}({\mos}_{0},\ub{\HUM})
\text{ by~\eqref{eq:model-control-vector} and 
\eqref{eq:campo-control-vector} } \\ 
>& g_{\hum}(\ub{\MOS},\ub{\HUM}) \text{ by~\eqref{eq:campo-control-vector} 
and } \ub{\MOS} < \mos_{0} \\ 
=& 0 \text{ by~\eqref{eq:carac_hat_mos}.} 
\end{align*}
Therefore, ${\hum}(t) > \ub{\HUM} $, for $t>0$ small enough.
As a consequence, we will concentrate on the case
$\hum_{0} < \ub{\HUM} $.
\medskip

Second, we show that there exists a point $(\dub{\mos}_{0}, \dub{\hum}_{0})$ 
such that
\begin{equation}
\dub{\mos}_{0} < {\mos}_{0} \text{ and } 
\dub{\mos}_{0} \in [\ub{\MOS},\MOS_{\infty}] \text{ and } 
\dub{\hum}_{0}=\hum_{0}=\solucion(\dub{\mos}_{0}) \leq \ub{\HUM} \eqfinp
\end{equation}
Of course, we take $\dub{\hum}_{0}=\hum_{0}$. 
To prove the existence of~$\dub{\mos}_{0} \in [\ub{\MOS},\MOS_{\infty}]$, 
it suffices to show that 
$\hum_{0} \in [\solucion(\MOS_{\infty}),\ub{\HUM}]$. Indeed, we know 
from Lemma~\ref{lema:solucion-frontera-derecha},
and especially~\eqref{eq:solucion_Y} that the function
$\solucion: [\ub{\MOS},\MOS_{\infty}]  \to [\solucion(\MOS_{\infty}),\ub{\HUM}]$ 
is $C^1$ and strictly decreasing. 
Now, from~\eqref{eq:constraintsetbackslashVV}, we deduce that 
\begin{itemize}
\item if $\MOS_{\infty}=1$, then 
\( \mos_{0} \leq \MOS_{\infty}=1 \) { and } 
\( \solucion(\mos_{0})<\hum_{0} \leq \ub{\HUM} \);
as the function~$\solucion$ is strictly decreasing, we conclude that
\(\solucion(\MOS_{\infty}) \leq \solucion(\mos_{0})<\hum_{0} \leq \ub{\HUM} \),
hence that $\hum_{0} \in [\solucion(\MOS_{\infty}),\ub{\HUM}]$;
\item if $\MOS_{\infty}<1$, then $\solucion(\MOS_{\infty})=0$ 
by~\eqref{eq:solucion_Y}; as $0 \leq \hum_{0}\leq \ub{\HUM} $,
we conclude that
\( 0=\solucion(\MOS_{\infty}) <\hum_{0} \leq \ub{\HUM} \),
hence that $\hum_{0} \in [\solucion(\MOS_{\infty}),\ub{\HUM}]$. 
\end{itemize} 
Notice in passing that, as the function~$\solucion$ is strictly decreasing, we 
have that
\begin{equation}
  \dub{\mos}_{0} = \ub{\MOS} \iff \dub{\hum}_{0}=\hum_{0}=\ub{\HUM} \eqfinp
\label{eq:in_passing}
\end{equation}
We have proved that $\dub{\mos}_{0}$ exists and that 
$\dub{\mos}_{0} \in [\ub{\MOS},\MOS_{\infty}]$. 
There remains to show that \( \dub{\mos}_{0} < {\mos}_{0} \).
Now, from $\dub{\mos}_{0} \in [\ub{\MOS},\MOS_{\infty}]$ and 
$ \dub{\hum}_{0}=\hum_{0}=\solucion(\dub{\mos}_{0}) $, we deduce that 
$(\dub{\mos}_{0},\dub{\hum}_{0})=(\dub{\mos}_{0},{\hum}_{0}) \in \VV$, 
defined in~\eqref{eq:nucleo-medio}. By definition~\eqref{eq:nucleo-medio} 
of~$\VV$, using the property that the function~$\solucion$ 
is strictly decreasing, we have that  
\[
(\dub{\mos}_{0},\dub{\hum}_{0}) =(\dub{\mos}_{0},{\hum}_{0}) \in \VV \Rightarrow 
({\mos},{\hum}_{0}) \in \VV \eqsepv \forall {\mos} \leq \dub{\mos}_{0}
\eqfinp
\]
As $({\mos}_{0},{\hum}_{0}) \not\in \VV$ by assumption~\eqref{eq:outsideVV},
we conclude that \( \dub{\mos}_{0} < {\mos}_{0} \).

From now on, we suppose that $\hum_{0}<\ub{\HUM} $, so that,
summing up the properties shown above,  
the point $(\dub{\mos}_{0}, \dub{\hum}_{0})$ is such that
\begin{equation}
\dub{\mos}_{0} < {\mos}_{0} \text{ and } 
\dub{\mos}_{0} \in ]\ub{\MOS},\MOS_{\infty}] \text{ and } 
\dub{\hum}_{0}=\hum_{0}=\solucion(\dub{\mos}_{0}) < \ub{\HUM} \eqfinp
\label{eq:sobre_la_frontera_derecha}
\end{equation}
\medskip

Now, in addition to the trajectory $ t \mapsto \bp{{\mos}(t), {\hum}(t)}$ 
that is the solution to~\eqref{eq:model-control-vector} for 
the control trajectory~$\controlm(\cdot)$ when 
$\bp{{\mos}(0), {\hum}(0)}=({\mos}_{0}, {\hum}_{0})$,
we study the trajectories
\begin{itemize}
\item $ t \mapsto \bp{\dub{\mos}(t), \dub{\hum}(t)}$,
the solution to~\eqref{eq:model-control-vector} when 
$\controlm(t)=\ub{\controlm}$ and when the starting point is 
$\bp{\dub{\mos}(0), \dub{\hum}(0)}=(\dub{\mos}_{0}, \dub{\hum}_{0})
=(\dub{\mos}_{0}, {\hum}_{0})$ 
as in~\eqref{eq:sobre_la_frontera_derecha}, that is, 
on the right of the frontier of~$\VV$ (the bottom
frontier of the shaded part in
Figure~\ref{fig:complemento-nucleo-medio});
\item $ t \mapsto \bp{\ub{\mos}(t), \ub{\hum}(t)}$,
the solution to~\eqref{eq:model-control-vector} when 
$\controlm(t)=\ub{\controlm}$ and when the starting point is 
$\bp{\ub{\mos}(0), \ub{\hum}(0)}=({\mos}_{0}, {\hum}_{0})$.
\end{itemize}
We have the following inequalities between initial conditions:
\[
\bp{\dub{\mos}(0), \dub{\hum}(0)}=(\dub{\mos}_{0}, \dub{\hum}_{0})
\leq ({\mos}_{0}, {\hum}_{0})=\bp{{\mos}(0), {\hum}(0)}
= \bp{\ub{\mos}(0), \ub{\hum}(0)} \eqfinp
\]
By Proposition~\ref{prop-cuasimonotono}, we deduce that 
\[
\bp{\dub{\mos}(t), \dub{\hum}(t)} \leq \bp{\ub{\mos}(t), \ub{\hum}(t)}
\leq \bp{{\mos}(t), {\hum}(t)} 
\eqsepv \forall t \geq 0 \eqfinp
\]
By Lemma~\ref{lema:frontera-derecha}, we obtain in particular that 
there exists $T> 0$ such that 
\[
\ub{\HUM} = \dub{\hum}(T) \leq \ub{\hum}(T) \leq {\hum}(T) \eqfinp
\]
Here, $T>0$ because 
(as can be seen in the proof of Lemma~\ref{lema:frontera-derecha}),
\[
\dub{\mos}(0)=\dub{\mos}_{0} > \ub{\MOS} \text{ and } 
\dub{\hum}(0)=\dub{\hum}_{0} < \ub{\HUM} \eqfinp
\]
We prove that $\dub{\hum}(T) < \ub{\hum}(T) $. 
\medskip

Suppose, by contradiction, that $\dub{\hum}(T) = \ub{\hum}(T) $. 
Therefore, the function \\
$ t \geq 0 \mapsto \ub{\hum}(t)-\dub{\hum}(t)$ is 
nonnegative with a minimum at~$T>0$, so that 
\[
\frac{d \ub{\hum}(t)}{dt}_{|t=T} - \frac{d \dub{\hum}(t)}{dt}_{|t=T} =0 \eqfinp
\]
From~\eqref{eq:model-control-vector}, we deduce that 
\[
A_{\hum} \ub{\mos}(T) (1-\ub{\hum}(T))-\gamma \ub{\hum}(T) =
A_{\hum} \dub{\mos}(T) (1-\dub{\hum}(T))-\gamma \dub{\hum}(T) \eqfinp
\]
As $\dub{\hum}(T) = \ub{\hum}(T) < 1$, we obtain that
$ \ub{\mos}(T)=\dub{\mos}(T)$.
Therefore, we have 
\[
\bp{\dub{\mos}(T), \dub{\hum}(T)} = \bp{\ub{\mos}(T), \ub{\hum}(T)}
\eqfinp 
\]
This is impossible because the two trajectories
$ t \mapsto \bp{\dub{\mos}(t), \dub{\hum}(t)}$,
and $ t \mapsto \bp{\ub{\mos}(t), \ub{\hum}(t)}$
are generated by the same vector fields, so that they must be equal.
However, we have 
\( \dub{\mos}(0)=\dub{\mos}_{0} < {\mos}_{0} = \ub{\mos}(0) \). 
\medskip

To conclude, we have proven that, for any control 
trajectory~$\controlm(\cdot)$ as 
in~\eqref{eq:restriccion-control}, the state trajectory starting 
from any point of the set~$\constraintset\backslash\VV$ does not satisfy the 
constraint~\eqref{eq:restriccion-estado} for a time~$T<+\infty$.
\end{proof}

\subsection{Epidemiological interpretation}

The following theorem summarizes the description of the viability kernel 
depending on whether the upper limit $\ub{\HUM}$ for the proportion 
of infected humans in~\eqref{eq:restriccion-estado} is low, high or medium. 

\begin{theorem}
\label{teo:nucleo-viabilidad}
The viability kernel~$\nucleo$ in~\eqref{eq:nucleo-viabilidad} is as follows.

\begin{description}
\item[L)] 
High constraint, that is, a low threshold of infected humans. If
\begin{equation}
\label{eq:restriccion-fuerte}
\ub{\HUM}<\frac{A_{\hum} - \gamma \ub{\controlm}/A_{\mos} }{A_{\hum} + \gamma }
\eqsepv
\end{equation}
the viability kernel consists only of the origin:
\begin{equation}
\nucleo=\{(0,0)\} \eqfinp
\end{equation}

\item[H)] Weak constraint, that is, a high threshold of infected humans. If
\begin{equation}
\label{eq:restriccion-debil}
\frac{A_{\hum}}{A_{\hum} + \gamma} < \ub{\HUM} \eqfinv
\end{equation}
the viability kernel~$\nucleo$ is the whole 
constraint set, that is,
\begin{equation}
\nucleo=\constraintset=\{(\mos, \hum)|0\leq\mos\leq 1, 0\leq\hum\leq \ub{\HUM} \}
=\rectangle\eqfinp
\end{equation}

\item[M)] Medium constraint, that is, a medium threshold of infected humans. 
If
\begin{equation}
\label{eq:restriccion-media}
0<\frac{A_{\hum} - \gamma \ub{\controlm}/A_{\mos} }{A_{\hum} + \gamma } < \ub{\HUM}< \frac{A_{\hum}}{A_{\hum} + \gamma} \eqfinv
\end{equation}
the viability kernel~$\nucleo$ is a strict subset 
of the constraint set~$\constraintset$
whose upper right frontier is a smooth and decreasing curve,
\begin{subequations} 
\begin{equation}
\nucleo=\left ([0, \ub{\MOS}]\times [0, \ub{\HUM}]\right )\bigcup
\left\{(\mos, \hum)\Big | \ub{\MOS}\leq \mos\leq \MOS_{\infty} \eqsepv 
\hum\leq \solucion(\mos)\right \}\eqfinv
\end{equation}
where $\solucion:m\mapsto \solucion(\mos)$ is the solution to
\begin{eqnarray}
\label{eq:frontera-superior-derecha}
-g_{\mos}(\mos, \solucion(\mos), \ub{\controlm})  
\solucion'(\mos)+g_{\hum}(\mos, \solucion(\mos))= 0\eqfinv \\
\solucion(\ub{\MOS})=\ub{\HUM} \eqfinp
\end{eqnarray}
\label{eq:nucleo-viable-medio}
\end{subequations}

\end{description}

\end{theorem}


The results of Theorem~\ref{teo:nucleo-viabilidad} make sense from an
epidemiological point of view. 
\begin{description}
\item[L)] 
One extreme situation is when the maximum proportion~$\ub{\HUM}$ 
of infected humans is low. In that case, 
the imposed constraint not to overshoot~$\ub{\HUM}$ can only be satisfied
when there are no infected humans nor infected mosquitoes at the start. 

\item[H)] 
The other extreme situation is when the maximum proportion~$\ub{\HUM}$ of 
infected humans is high. In that case, we are allowing a large fraction of 
the population to get infected. For any trajectory starting in the so-called 
constraint set (that is, such that the value with which the proportion of 
people starts is located below the maximum proportion), any 
mortality rate due to fumigation will make the proportion of infected
 humans always below~$\ub{\HUM}$ for all times.

\item[M)] 
The most interesting case 
is when the maximum proportion~$\ub{\HUM}$ of infected humans is medium 
(satisfying the condition~\eqref{eq:restriccion-media}). Here, 
the viability kernel is the strict subset of the constraint set whose upper 
frontier matches the state orbit associated with the maximum mortality rate 
due to fumigation (by Lemma~\ref{lema:frontera-derecha}).
\end{description}
Focusing more on this last case, the 
condition~\eqref{eq:restriccion-media} of Theorem~\ref{teo:nucleo-viabilidad}
can be rewritten as follows: 
\begin{equation}
\label{eq:desigualdad-caso-medio}
0< \frac{A_{\hum}}{\gamma+A_{\hum}}-\frac{\gamma}{A_{\mos}(\gamma+A_{\hum})} 
\ub{\controlm}<
\ub{\HUM} < \frac{A_{\hum}}{\gamma+A_{\hum}}\eqfinp
\end{equation}

This new form~\eqref{eq:desigualdad-caso-medio} taken by the 
inequalities~\eqref{eq:restriccion-media} allows us to visualize the 
relation between the maximum proportion~$\ub{\HUM}$ of infected humans 
and the maximum~$\ub{\controlm}$ mortality fumigation rate. 
Figure~\ref{fig:umbral-control} displays the graph of
$\ub{\HUM}$ versus $\ub{\controlm}$. 

\begin{figure}
\centering
\includegraphics[width=13cm,height=7.7cm]{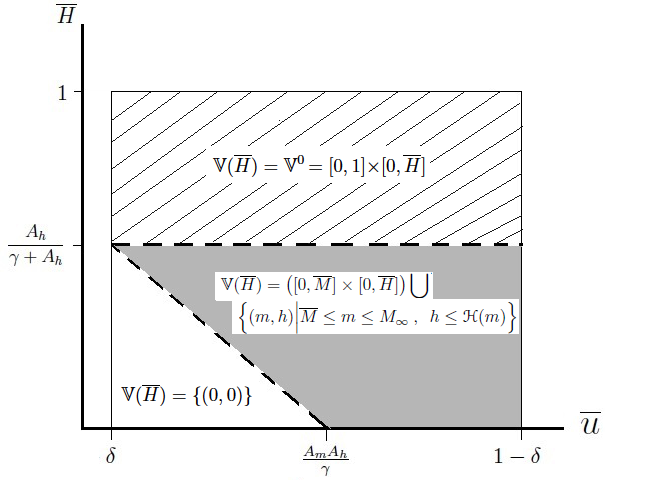}
\caption{Viability kernel for the Ross-MacDonald 
model~\eqref{eq:model-Ross-Macdonald}, 
in function of the cap~$\ub{\HUM}$ on the 
proportion of infected humans and of the maximum mortality 
rate~$\ub{\controlm}$ of mosquitoes}
\label{fig:umbral-control}
\end{figure}

Equalities in the inequalities on both sides 
of~\eqref{eq:desigualdad-caso-medio}
correspond to the two dashed 
straight lines on Figure~\ref{fig:umbral-control}. 
\begin{description}
\item[L)] 
When the couple $(\ub{\controlm}, \ub{\HUM})$ belongs to the 
unshaded bottom triangle 
(corresponding to a low upper bound~$\ub{\HUM}$), 
the viability kernel consists only of the origin. 
\item[H)] 
When the couple $(\ub{\controlm}, \ub{\HUM})$ belongs to the 
upper rectangle shaded with lines 
(corresponding to a high upper bound~$\ub{\HUM}$), 
the viability kernel is the 
whole constraint set $\constraintset=\rectangle$. 
\item[H)] 
When the couple $(\ub{\controlm}, \ub{\HUM})$ belongs to the 
the gray shaded region (corresponding to a medium upper bound~$\ub{\HUM}$),
the viability kernel is a strict subset of the constraint 
set~$\constraintset=\rectangle$, 
whose top right frontier is the smooth curve given 
in~\eqref{eq:frontera-superior-derecha}. 
\end{description}

\subsection{Viable controls}
\label{sec:controles-viables}

After having characterized the viability kernel, 
we turn to discussing so-called \emph{viable controls}.
In viability theory \cite{Aubin1991}, it is proven that any 
control $\controlm \in [\lb{\controlm},\ub{\controlm}]$, 
such that the vector~$\campo$ points towards the inside 
of the viability kernel~$\nucleo$ is \emph{viable} in the following sense:
a trajectory of viable controls and an initial state in~$\nucleo$
produce a state trajectory in~$\nucleo$.

\begin{description}
\item[L)] 
With a strong constraint as in~\eqref{eq:restriccion-fuerte},
it is impossible to keep the proportion of infected humans below~$\ub{\HUM}$
for all times, for \emph{whatever} admissible control 
trajectory~\eqref{eq:restriccion-control}
(except if there are no infected humans or mosquitoes at the start), 
by Proposition~\ref{prop:restriccion-fuerte}.
\item[H)] 
With a weak constraint as in~\eqref{eq:restriccion-debil}, 
the constraint set $\constraintset=[0, 1]\times[0, \ub{\HUM}]$ is strongly
invariant, by Proposition~\ref{prop:restriccion-debil}.
Therefore, \emph{any} admissible control 
trajectory~\eqref{eq:restriccion-control} makes it possible that
the proportion of infected humans remains below~$\ub{\HUM}$
for all times. For instance, the minimum control $\lb{\controlm}=\delta$,
corresponding to natural death rate (hence to no control), is a viable
control.
\item[M)] 
With a medium constraint as in~\eqref{eq:restriccion-media}, 
a viable control is the stationary control $\controlm(\cdot)=\ub{\controlm}$,
corresponding to maximal fumigation. However, other controls might be viable:
it suffices that, on the frontier of the viability kernel~$\nucleo$, 
a control yields a vector~$\campo$ that points towards the inside of~$\nucleo$.
For instance, we could opt for a control which is low when far from the frontier
of the viability kernel~$\nucleo$, 
but reaches~$\ub{\controlm}$ when close to the frontier.
Letting $d(\mos,\hum)$ be the distance between point~$(\mos,\hum)$
and the upper frontier of the viability kernel~$\nucleo$, 
we propose the following viable control for 
problem~\eqref{eq:problema-viabilidad-recursos-ilimitados}:
\begin{equation}
\controlm(t)=\Big( 1-\exp\Big(-d\big(\mos(t),\hum(t)\big)\Big) \Big) 
\lb{\controlm}
+ \exp\Big(-d\big(\mos(t),\hum(t)\big)\Big) \ub{\controlm} \eqfinp
\label{eq:control-viable}
\end{equation}
Figures~\ref{fig:controles-viables.jpg} and \ref{fig:infected-viables.jpg}
display examples of viable control and state (infected humans) 
trajectories for the Ross-MacDonald model~\eqref{eq:model-Ross-Macdonald} 
with a proportion of 1\% infected humans not to be exceeded ($\ub{\HUM}=0.01$).
\end{description}

\begin{figure}
\centering
\includegraphics[width=10cm,height=6cm]{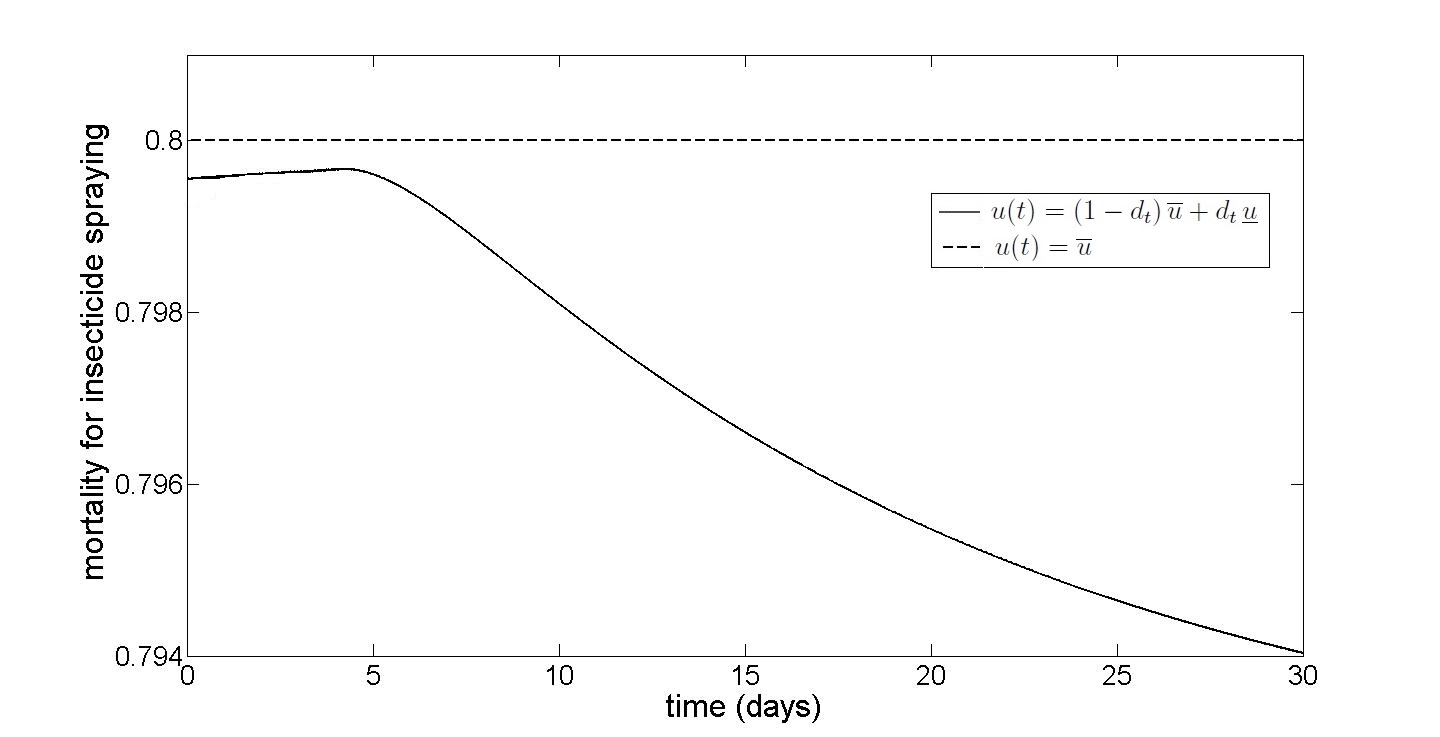}
\caption{Viable control trajectories for the Ross-MacDonald 
model~\eqref{eq:model-Ross-Macdonald} with $\ub{\HUM}=0.01$
\label{fig:controles-viables.jpg}}
\end{figure}

\begin{figure}
\centering
\includegraphics[width=10cm,height=6cm]{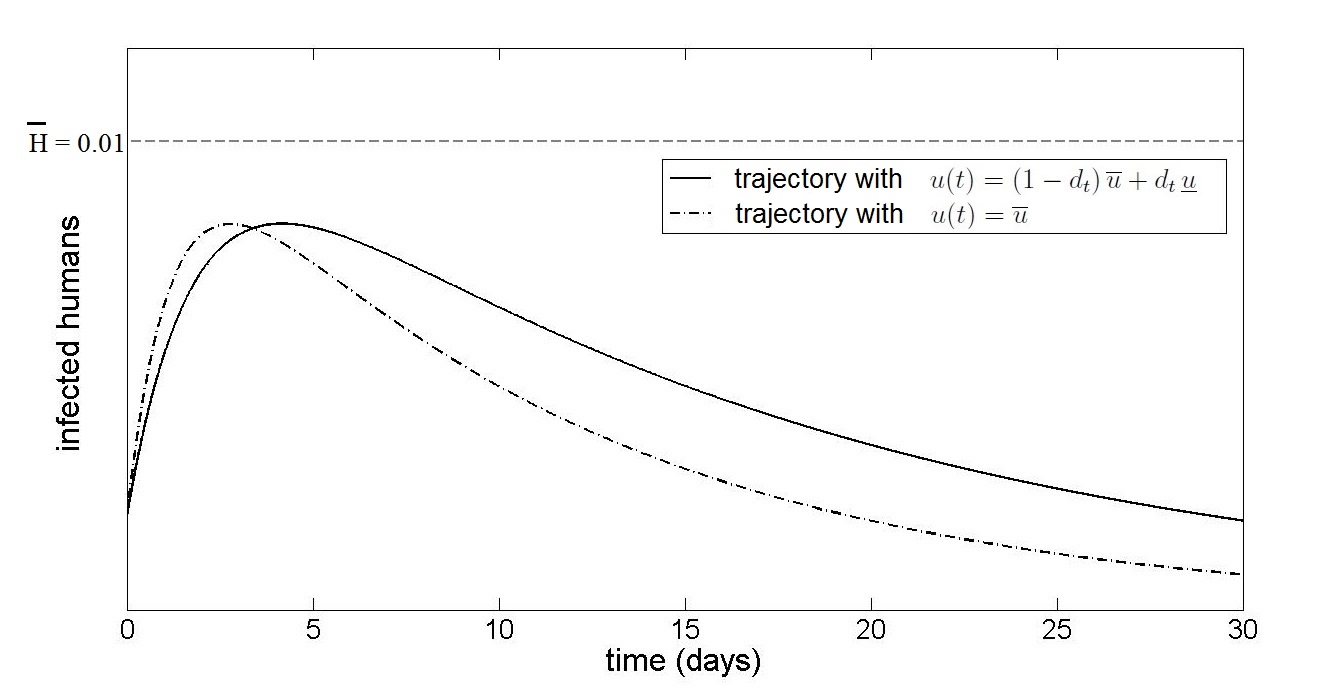}
\caption{Viable state (infected humans) trajectories for the Ross-MacDonald 
model~\eqref{eq:model-Ross-Macdonald} with $\ub{\HUM}=0.01$
\label{fig:infected-viables.jpg}}
\end{figure}

\section{Conclusion}
\label{sec:conclusion-viabilidad}

As said in the introduction, the approach we have developed 
is (to our best knowledge) new in mathematical epidemiology.
Instead of aiming at an equilibrium or optimizing, we have looked for policies
able to maintain the infected individuals below a threshold for all times. 
More precisely, we have allowed for time-dependent fumigation rates 
to reduce the population
of mosquito vector, in order to maintain the proportion of infected 
individuals by dengue below a threshold for all times.
By definition, the viability kernel is
the set of initial states (mosquitoes and infected individuals) for which  
such a fumigation control trajectory exists.

Our theoretical results are the following. 
For the Ross-Macdonald model with vector mortality control,
we have been able to characterize the viability kernel associated with 
the viability constraint that consists in capping 
the proportion of infected humans for all times. 
Depending on whether the cap on the proportion of infected is 
low, high or medium, we have provided different expressions 
of the viability kernel.
We have also characterized so-called viable policies that produce, 
at each time, a fumigation rate as a function of 
current proportions of infected humans and mosquitoes,
such that the proportion of infected humans remains 
below a threshold for all times. 

Regarding the use of our theoretical results,
we have provided a numerical application in the case of the control
of a dengue outbreak in 2013 in Cali, Colombia.  
Indeed, thanks to numerical data yielded by the
Municipal Secretariat of Public Health of Cali, we have produced figures 
of viability kernels and of viable trajectories. 
What our analysis suggests is that, to cap the proportion of infected humans
at the peak, you need to measure the proportion of infected mosquitoes,
at least when it is larger than the characteristic proportion~$\ub{\MOS}$,
defined in~\eqref{eq:hat_mos}. 
Measuring mosquito abundance is a difficult task, whereas measuring
a \emph{proportion} of infected mosquitoes might be done by proper sampling.
Naturally, the weaker the control, 
the more you need to estimate the proportion of infected mosquitoes,
as illustrated in Figure~\ref{fig:nucleo-2013}. 
\begin{figure}
\centering
\includegraphics[width=13cm,height=7.7cm]{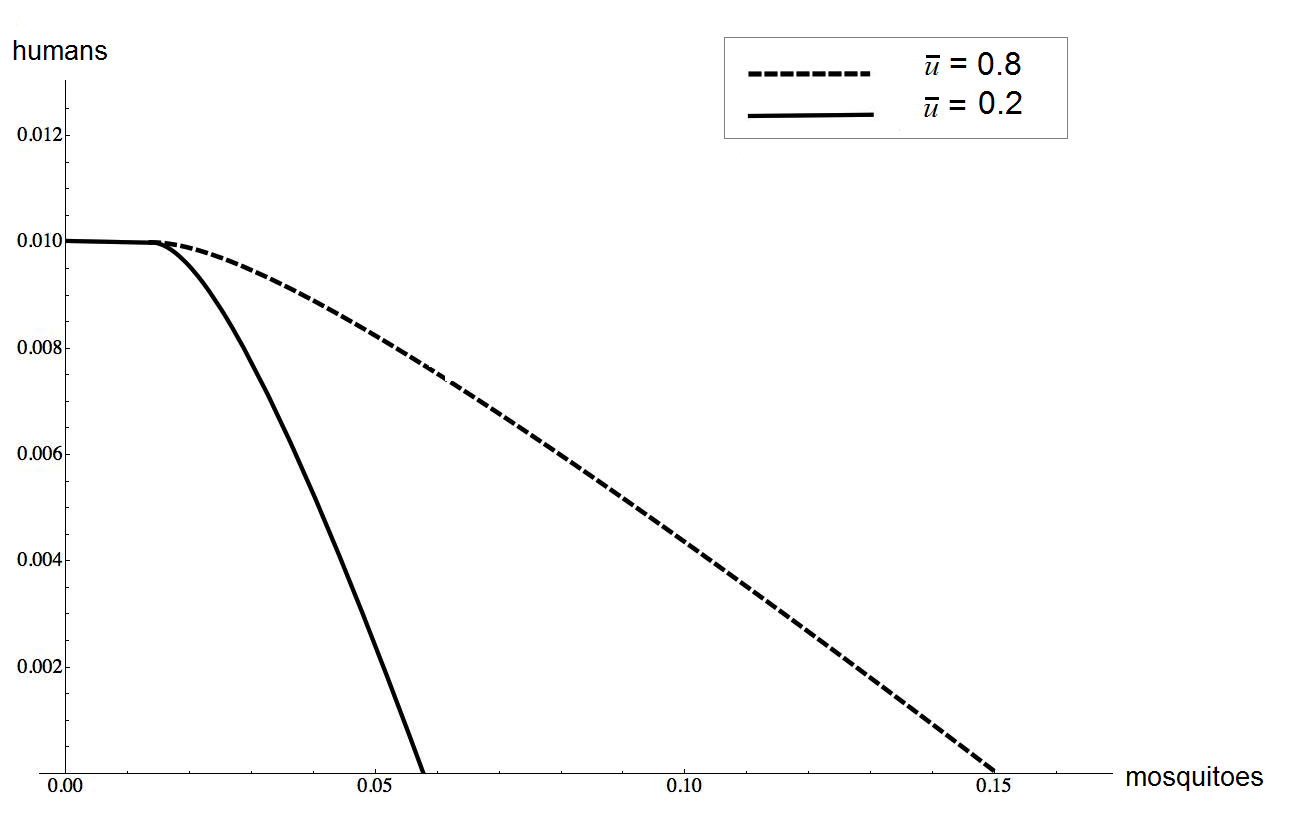}
\caption{Dependence of the viability kernel on the maximal fumigation 
mortality rate~$\ub{\controlm}$}
\label{fig:nucleo-2013}
\end{figure}


\appendix 

\section{Appendix}

\subsection{Recalls on comparison theorems for differential systems}
\label{sec:resultados-teoricos}

\begin{definition}
\label{def:funcion-cuasimonotona}
The function $g:\RR^n\rightarrow \RR^n$ is said to be 
\emph{quasi monotonous} if $g$ is $C^1$ and that 
 \[
 \frac{\partial g^i}{\partial x_j}\geq 0 \eqsepv \forall i \neq j \eqfinp
 \]
\end{definition}

In what follows, all inequalities between vectors 
have to be understood componentwise. 
\begin{theorem}
\label{teo:comparacion-edo}
Let $f$ and $g$ be two vector fields on $D\subset\mathds{R}^n$, 
with $f$ or $g$ {quasi monotonous} and $f\leq g$.
Suppose that the differential systems
\[
\dot{\xx}=f(\xx) \eqsepv \dot{\yy}=g(\yy)\eqfinv
\]
have solutions $t \mapsto \xx_t$ and $t \mapsto \yy_t$ defined 
for all $t \geq 0$. Then, 
if $\xx_0\leq \yy_0$, we have that $\xx_t\leq \yy_t$ for all $t \geq 0$.
\end{theorem}

\begin{proof}
We consider two cases. 
\begin{enumerate}[(a)]
\item 
Suppose that $f<g$ --- that is, $f^i<g^i$ for all $i=1,\dots,n$ ---
and that $\xx_0<\yy_0$. We define 
\[
\tau=\inf \{t\geq 0 \mid \exists i=1,\dots,n \eqsepv \xx_t^i>\yy_t^i\}\eqfinp
\]
We show that $\tau=+\infty$, that is, $\xx_t\leq \yy_t$ for all $t \geq 0$.
Indeed, let us suppose the contrary. 
If $\tau<+\infty$, then there exists at least one $i=1,\dots,n$ such that
\[
\xx^i_{\tau}=\yy^i_{\tau} \eqsepv \xx^j_{\tau}\leq \yy^j_{\tau}
\eqsepv \forall j \neq i \eqfinp
\]
Supposing that $g$ is {quasi monotonous} (the proof is similar
when $f$ is {quasi monotonous}), 
we deduce that $g^i(\xx_{\tau})\leq g^i(\yy_{\tau})$.

Moreover, as $\xx_t^i>\yy_t^i$, for all $t\in ]\tau, \tau+\epsilon[$
(for $\epsilon$ small enough) 
and $\xx^i_{\tau}=\yy^i_{\tau} $, we have that 
\[
f^i(\xx_{\tau}) = \frac{d \xx_t^i}{dt}_{|t=\tau}
\geq \frac{d \yy_t^i}{dt}_{|t=\tau} = g^i(\yy_{\tau}) \eqfinp
\]
From $g^i(\xx_{\tau})\leq g^i(\yy_{\tau})$, $ g^i(\yy_{\tau}) \leq f^i(\xx_{\tau}) $
and $f^i<g^i$, we deduce that
\[
g^i(\xx_{\tau})\leq g^i(\yy_{\tau})\leq f^i(\xx_{\tau})<g^i(\xx_{\tau})\eqfinp
\]
This is contradictory. As a consequence, 
$\xx_0< \yy_0$ implies that $\xx_t\leq \yy_t$ for all $t \geq 0$.

\item 
Suppose that $f \leq g$ and $\xx_0 \leq \yy_0$.
For any $\epsilon >0$, denote by $\yy_t^{\epsilon}$ the solution of 
\[
\dot{\yy^{\epsilon}}=g(\yy^{\epsilon})+\epsilon
\eqsepv \yy^{\epsilon}_0=\yy_0+\epsilon\eqfinp
\]
We have that 
\[
\xx_0<\yy^{\epsilon}_0 \eqsepv f<g+\epsilon\eqfinp
\]
Therefore, we can conclude from the previous item that
$\xx_t\leq \yy^{\epsilon}_t$ for all $t \geq 0$.
It is well known that, for fixed~$t$, when $\epsilon\downarrow 0$, we have that 
$\yy^{\epsilon}_t\rightarrow \yy_t$.

As a consequence, 
$\xx_0\leq \yy_0$ implies that $\xx_t\leq \yy_t$ for all $t \geq 0$.
\end{enumerate}

\end{proof}

The following extension is immediate.

\begin{theorem}
\label{teo:comparacion-edo-extension}
Let $(f_t)_{t\geq 0}$ and $(g_t)_{t\geq 0}$ be two families
of vector fields on $D\subset\mathds{R}^n$, 
with $f_t$ or $g_t$ {quasi monotonous} and $f_t\leq g_t$,
for all~$t\geq 0$.
Suppose that the differential systems
\[
\dot{\xx}=f_t(\xx) \eqsepv \dot{\yy}=g_t(\yy)\eqfinv
\]
have solutions $t \mapsto \xx_t$ and $t \mapsto \yy_t$ defined 
for all $t \geq 0$. Then, 
if $\xx_0\leq \yy_0$, we have that $\xx_t\leq \yy_t$ for all $t \geq 0$.
\end{theorem}

%
%
%
%
%
%
%
%
%
%
%

\subsection{Epidemic model adjusted to 2013 dengue outbreak in Cali, Colombia}

\def\zz{z}
\def\thetab{\theta}
The city of Cali, in Colombia, has witnessed several episodes of dengue,
as displayed in Figure~\ref{fig:brotes-cali}. 
\begin{figure}
\centering
\includegraphics[width=13cm,height=7.7cm]{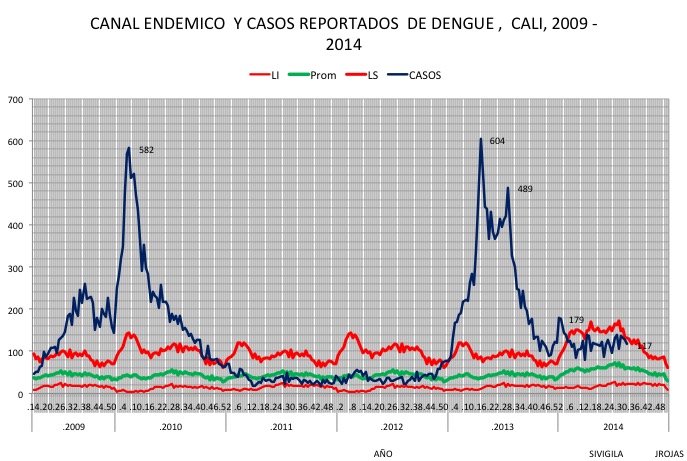}
\caption{Number of infected by dengue, revealing several episodes of dengue 
in the city of Cali, in Colombia}
\label{fig:brotes-cali}
\end{figure}
The Municipal Secretariat of Public Health of Cali provided us with data
corresponding to the dengue outbreak registered in 2013.
Here, we present how we have estimated the parameters in 
the Ross-Macdonal model~\eqref{eq:model-Ross-Macdonald} ---
using the information of daily reports for new registered cases of dengue in Cali --- to obtain different figures of viability kernels and of viable trajectories.

\subsubsection*{Definition of parameters} 

We introduce the state vector 
\begin{equation}
  \zz(t)= \big( \mos(t), \hum(t) \big) \in [0,1]^2 \eqfinv
\end{equation}
and the vector of parameters
\begin{equation}
\thetab=\big( \alpha, p_{\hum}, p_{\mos}, \xi, \delta \big) \in \Theta 
 \subset \mathbb{R}^5_+  \eqfinv
\label{eq:parameter}
\end{equation}
consisting of the five parameters previously defined in 
Table~\ref{tabla:parametros-Ross}. The parameter set 
$\Theta \subset \mathbb{R}^5_+$ is given by the Cartesian product of the 
five intervals in the third column of Table~\ref{tab-A1}. 

With these notations, the Ross-Macdonal model~\eqref{eq:model-Ross-Macdonald}
now writes
\begin{equation}
\label{RMmodel-par}
\begin{array}{rclcl}
\dfrac{d}{dt} \zz(t;\thetab) & = & \ff \big( \zz(t;\thetab),\thetab \big) & & t> 0\\
            \zz(0; \thetab)    & = & \zz_0                                  & & \zz_0 = ( \mos_0, \hum_0)^{\prime}
\end{array}
\end{equation}
where
\begin{equation}
\ff (\zz,\thetab) = \left( \begin{array}{c}
f_1 (\zz,\thetab) \\ f_2 (\zz,\thetab)
\end{array} \right)  = \left( \begin{array}{c}
\alpha\, p_{\mos}\,\hum (1-\mos) - \delta\, \mos \\ \alpha\, p_{\hum}\, \xi \, \mos (1-\hum)-\gamma\, \hum
\end{array} \right) \eqfinv \gamma=0.1 \eqfinp 
\end{equation}

\subsubsection*{Daily data deduced from health reports}

Notice that the rate~$\gamma$ of human recovery 
does not appear in the parameter vector~$\thetab$
in~\eqref{eq:parameter}. Indeed, in the data we only have new cases of dengue 
registered per day; there is no information regarding how many people recover 
daily. We chose an infectiousness period of~10 days, that is, 
a rate of human recovery fixed at $\gamma=0.1$. 
Under this assumption, the \emph{daily incidence data} 
(i.e., numbers of newly registered cases reported on daily basis) 
provided by the Municipal Secretariat of Public Health (Cali, Colombia) 
can be converted into the \emph{daily prevalence data}
(i.e., numbers of infected people on a given day, be they new or not).
With this, we deduce values of daily proportion of infected people 
in the form of the set
\begin{equation}
\label{dataset}
\mathbb{O} = \Big\{ \big( t_j, \widehat{\hum}_j \big) \eqsepv 
j =0,1, \ldots, \mathcal{D} \Big\} \eqfinv 
\end{equation}
where $t_j$ refers to $j$-th day, within the observation period 
of~$(\mathcal{D}+1)$ days, and where $\widehat{\hum}_j$ stands for the 
fraction of infected people at the day~$t_j$. Naturally, the first couple 
in the set~\eqref{dataset} defines the initial condition $\hum(t_0)=\widehat{\hum}_0$ with $t_0=0$ (initial observation day). 
Unfortunately, there is no available data for the fraction of infected 
mosquitoes. As mosquito abundance is strongly correlated with dengue outbreaks \cite{Jansen2010}, we have chosen a linear relation 
$\mos(0)= 3 \widehat{\hum}_0$ at the beginning of an epidemic outburst
(other choices gave similar numerical results).

\subsubsection*{Parameter estimation procedure}

To estimate a parameter vector~$\thetab \in \Theta$ that suits with the data, 
we have applied the curve-fitting approach based on least-square method.
More precisely, we look for an optimal solution to 
the problem of constrained optimization 
\begin{equation}
\label{min}
\min \limits_{\thetab \in \Theta} \varphi (\thetab) = \frac{1}{2} \sum \limits_{j=1}^\mathcal{D}\, \big( \hum(t_j,\thetab)-\widehat{\hum}_j \big)^2,  \quad \mathcal{D}=60 \; \mbox{days} \eqfinv
\end{equation}
subject to the differential constraint \eqref{RMmodel-par}. 
Regarding numerics, we have solved this optimization problem with the 
\texttt{lsqcurvefit} routine (MATLAB Optimization Toolbox),
starting with an admissible $\thetab_0 \in \Theta$ (see its exact value in Table~\ref{tab-A1}, second column). The routine generates a sequence 
$\{ \thetab_1, \thetab_2, \ldots \}$ that we stop once it is stationary, 
up to numerical precision.
For a better result, we have combined two particular methods 
(Trust-Region-Reflective Least Squares Algorithm \cite{Sorensen1982} and 
Levenberg-Marquardt Algorithm \cite{More1978}) in the implementation of 
the \texttt{lsqcurvefit} MATLAB routine.

\begin{table}[t]
\begin{center}
\begin{tabular}{|c|c|c|c|c|}\hline
Parameter &  Initial value & Range & Reference & Estimated value \\ \hline

$\alpha$            &  1     & $[0, 5]$ & \cite{Costero1998,Scott2000b} & 0,3365 \\ \hline
$p_{\mos}$          &  0.5   & $[0, 1]$ &  & 0,1532 \\ \hline
$p_{\hum}$          &  0.5   & $[0, 1]$ &  & 0,2287 \\ \hline
$\xi$               &  1     & $[1, 5]$ & \cite{Mendez2006,Scott2000a} & 1,0359 \\ \hline
$\delta$            &  0.035 & $\left[ \frac{1}{30}, \frac{1}{15} \right]$ & \cite{Costero1998,Scott2000a} & 0.0333 \\ \hline
\end{tabular}
\end{center}
\caption{Initial values, admissible ranges, respective source references, and estimated values of parameters (numerical solution of optimization problem \eqref{min}).}
\label{tab-A1}
\end{table}

The last column of Table~\ref{tab-A1} provides estimated values for the
parameters $\thetab= \big( \alpha, p_{\hum}, p_{\mos}, \xi, \delta \big)$,
and Figure~\ref{fig-A1} displays the curve-fitting results. 
  \begin{figure}
  	\begin{center}
	   \includegraphics[width=14cm,height=7cm]{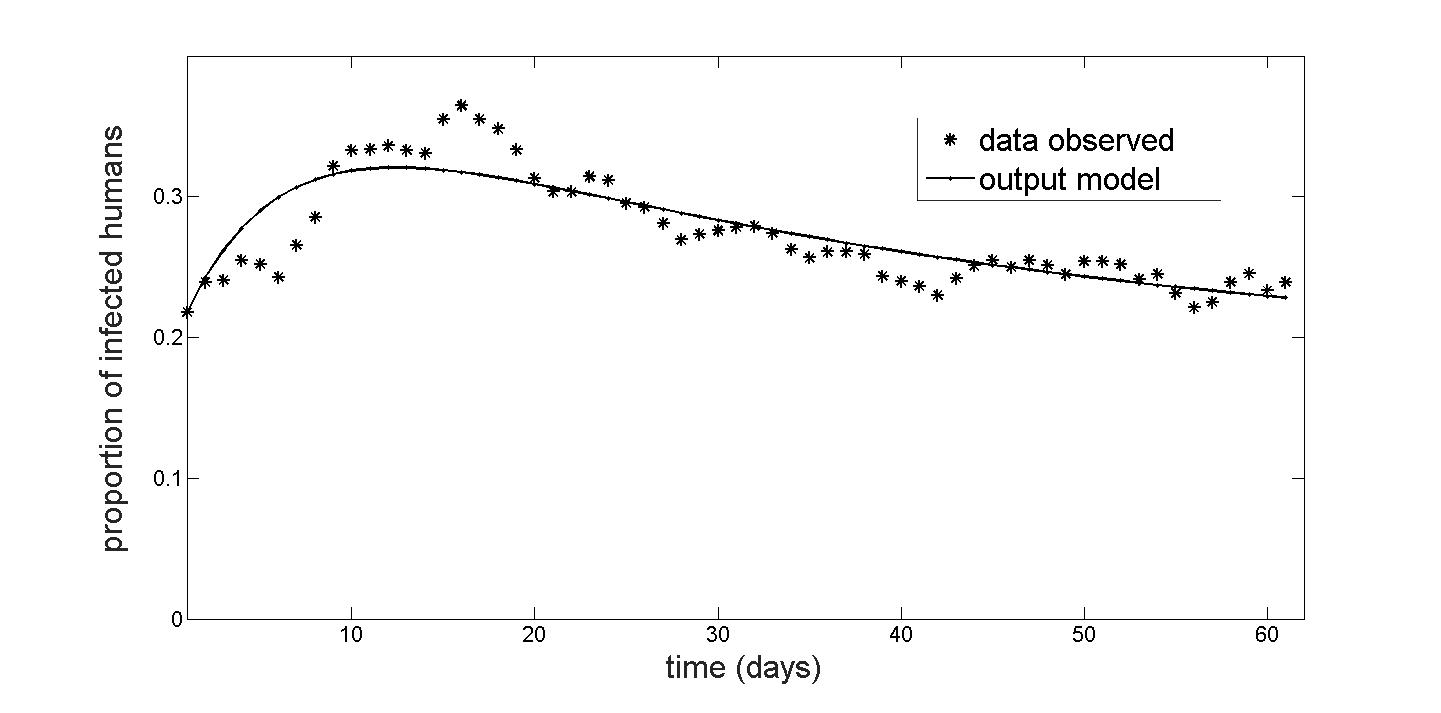}
	\end{center}
	\caption{Fraction of people infected with dengue obtained by 
adjustment of the Ross-Macdonal model~\eqref{eq:model-Ross-Macdonald} 
(smooth solid curve) versus registered daily prevalence cases 
(star isolated points) during the 2013 dengue outbreak in Cali, Colombia
\label{fig-A1}}
  \end{figure}

\paragraph*{Acknowledgments.}

The authors thank the French program PEERS-AIRD 
(\emph{Modèles d'optimisation et de viabilité en écologie et en économie})
and the Colombian Programa Nacional de Ciencias Básicas COLCIENCIAS
(\emph{Modelos y métodos matemáticos para el co trol y vigilancia del dengue},
código 125956933846)
that offered financial support for missions, together with 
\'Ecole des Ponts ParisTech (France), 
Universit\'e Paris-Est (France), 
Universidad Aut\'onoma de Occidente (Cali, Colombia)
and Universidad del Valle (Cali, Colombia).

\newcommand{\noopsort}[1]{} \ifx\undefined\allcaps\def\allcaps#1{#1}\fi

\end{document}